\documentclass[amscd, amssymb,11pt]{article}
\usepackage[usenames]{color}\definecolor{red}{rgb}{1.0,0.0,0.0}

\definecolor{blu}{rgb}{0.0,0.0,1.0}

\definecolor{gre}{rgb}{0.03,0.50,0.03}

\normalbaselines
\usepackage{amsmath}
\usepackage{amsthm}
\usepackage{amsfonts}
\usepackage{amssymb}
\usepackage{bbm}
\usepackage{wrapfig}
\usepackage{graphicx}
\usepackage{multimedia}

\newtheorem{theorem}{Theorem}[section]
\newtheorem{lemma}[theorem]{Lemma}

\newtheorem{corollary}[theorem]{Corollary}
\newtheorem{remark}[theorem]{Remark}
\newtheorem{definition}[theorem]{Definition}

\let\Section=\section
\def\section{\setcounter{equation}{0}\Section}
\sf

\def\Swiech
{{\accent"13S}wie{\hbox{\kern -0.21em\lower
0.79ex\hbox{$\textfont1=\scriptfont1
\lhook$}}}ch}
\def\SWIECH
{{\accent"13S}WIE{\hbox{\kern -0.26em\lower
0.77ex\hbox{$\textfont1=\scriptfont1
\lhook$}}}CH}

\begin{document}

\title{{\bf Perron's method for nonlocal fully nonlinear equations
%\footnote{ XXXX}}
}}

\author{
    \textsc{Chenchen Mou}\\
    \textit{Department of Mathematics, UCLA}\\
\textit{
Los Angeles, CA 90095, U.S.A.}\\
 \textit{E-mail: muchenchen@math.ucla.edu}   
  }
\date{}

\maketitle

\begin{abstract}
This paper is concerned with existence of viscosity solutions of non-translation invariant nonlocal fully nonlinear equations. We construct a discontinuous viscosity solution of such nonlocal equation by Perron's method. If the equation is uniformly elliptic, we prove the discontinuous viscosity solution is H\"older continuous and thus it is a viscosity solution.
\end{abstract}

\vspace{.2cm}

\vspace{.2cm}
\section{Introduction}
In this paper, we investigate existence of a viscosity solution of
\begin{equation}\label{eq:gennon}
\left\{\begin{array}{ll} I(x,u(x),u(\cdot))=0,\quad \text{in $\Omega$,}\\
\qquad\qquad\quad\,\,\,\,\, u=g,\quad \text{in $\Omega^c$},
\end{array}
  \right.
\end{equation}
where $\Omega$ is a bounded domain in $\mathbb R^n$, $I$ is a non-translation invariant nonlocal operator and $g$ is a bounded continuous function in $\mathbb R^n$. 

An important example of $(\ref{eq:gennon})$ is the Dirichlet problem for nonlocal Bellman-Isaacs equations, i.e.,
\begin{equation}\label{eq:belisa}
\left\{\begin{array}{ll} \sup_{a\in\mathcal{A}}\inf_{b\in\mathcal{B}}\{-I_{ab}[x,u]+b_{ab}(x)\cdot \nabla u(x)+c_{ab}(x)u(x)+f_{ab}(x)\}=0,
\quad \text{in $\Omega$},\\
\qquad\qquad\qquad\qquad\qquad\qquad\quad\qquad\qquad\qquad\qquad\qquad\qquad\quad\quad\,\, u=g,\quad \text{in $\Omega^c$},
\end{array}
  \right.
\end{equation}
where $\mathcal{A}, \mathcal{B}$ are two index sets, $b_{ab}:\mathbb R^n\to\mathbb R^n$, $c_{ab}:\mathbb R^n\to \mathbb R^+$, $f_{ab}:\mathbb R^n\to\mathbb R$ are uniformly continuous functions and $I_{ab}$ is a L\'evy operator. If the L\'evy measures are symmetric and absolutely continuous with respect to the Lebesgue measure, then they can be represented as
\begin{equation}\label{eq:levope}
I_{ab}[x,u]:=\int_{\mathbb R^n}[u(x+z)-u(x)]K_{ab}(x,z)dz,
\end{equation}
where $\{K_{ab}(x,\cdot);x\in\Omega,a\in\mathcal{A},b\in\mathcal{B}\}$ are kernels of L\'evy measures satisfying
\begin{equation}\label{eq:levmea}
\int_{\mathbb R^n}\min\{|z|^2,1\}K_{ab}(x,z)dz<+\infty\quad \text{for all $x\in\Omega$}.
\end{equation}
In fact, we will not assume our L\'evy measures to be symmetric in the following sections. 

Existence of viscosity solutions has been well established for the Dirichlet problem for integro-differential equations by Perron's method when the equations satisfy the comparison principle. In \cite{BI}, G. Barles and C. Imbert studied the comparison principle for degenerate second order integro-differential equations assuming the nonlocal operators are of L\'evy-It\^o type and the equations satisfy the coercive assumption. Then G. Barles, E. Chasseigne and C. Imbert obtained existence of viscosity solutions for such integro-differential equations by Perron's method in \cite{BCI0}. L. A. Caffarelli and L. Silvestre proved, in Section 5 of \cite{LL1}, the comparison principle for uniformly elliptic translation invariant integro-differential equations where the nonlocal operators are of L\'evy type. Then existence of viscosity solutions follows, if suitable barriers can be constructed, by Perron's method. Later H. Chang Lara and G. Davila extended the comparison and existence results of \cite{LL1} to parabolic equations, see Section 3 in \cite{CD1,CD4}. The existence for $(\ref{eq:gennon})$ when $I$ is a non-translation invariant nonlocal operator is much more difficult to tackle since we do not have a good comparison principle, see \cite{MS}. In \cite{MS}, the authors proved comparison assuming that either a viscosity subsolution or a supersolution is more regular.  To our knowledge, the only available results for existence of solutions for non-translation invariant equations are the following. D. Kriventsov studied, in Section 5 of \cite{Kri}, existence of viscosity solutions of some uniformly elliptic nonlocal equations. In Section 4 of \cite{Se1}, J. Serra proved existence of viscosity solutions of uniformly elliptic nonlocal Bellman equations. H. Chang Lara and D. Kriventsov extended existence results in \cite{Kri} to a class of uniformly parabolic nonlocal equations, see Section 5 of \cite{CK}. In all these proofs, the authors used fixed point arguments. In \cite{AT}, O. Alvarez and A. Tourin obtained existence of viscosity solutions of degenerate parabolic nonlocal equations by Perron's method with a restrictive assumption that the L\'evy measures are bounded. The boundedness of L\'evy measures allows them to obtain the comparison principle. The reader can consult \cite{CIL,Is,Is1,Ko} for Perron's method for viscosity solutions of fully nonlinear partial differential equations. 

The probability literature on existence of viscosity solutions of nonlocal Bellman-Isaacs equations is enormous. It is well-known that Bellman-Isaacs equations arise when people study the differential games, where the equations carry information about the value and strategies of the games. The probabilists represent viscosity solutions of nonlocal Bellman-Isaacs equations as value functions of certain stochastic differential games with jump diffusion via the dynamic programming principle. However, mostly in the probability literature, the nonlocal terms of nonlocal Bellman-Isaacs equations are of L\'evy-It\^o type and $\Omega$ is the whole space $\mathbb R^n$. We refer the reader to \cite{BBP10,B10,B11,BHL10,I10,KP10,KS10,OS10,P10,S10,S11,SZ10} for stochastic representation formulas for viscosity solutions of nonlocal Bellman-Isaacs equations.

In Section 3, we adapt to the nonlocal case the approach from \cite{Is,Is1,Ko} for obtaining existence of a discontinuous viscosity solution $u$ of $(\ref{eq:gennon})$ without using the comparison principle. For applying Perron's method, we need to assume that there exist a continuous viscosity subsolution and a continuous supersolution of $(\ref{eq:gennon})$ and both satisfy the boundary condition. Since $(\ref{eq:gennon})$ involves the nonlocal term, the proof of the existence is more delicate than the PDE case.

In Section 4, we obtain a H\"older estimate for the discontinuous viscosity solution of $(\ref{eq:gennon})$ constructed by Perron's method assuming the equation is uniformly elliptic. In most of the literature, the nonlocal operator $I$ is assumed to be uniformly elliptic with respect to a class of linear nonlocal operators of form $(\ref{eq:levope})$ with kernels $K$ satisfying
\begin{equation}\label{eq:uell1}
(2-\sigma)\frac{\lambda}{|z|^{n+\sigma}}\leq K(x,z)\leq (2-\sigma)\frac{\Lambda}{|z|^{n+\sigma}},
\end{equation}
where $0<\lambda\leq \Lambda$. Various of regularity results were obtained in recent year under the above uniform ellipticity such as \cite{LL1,LL2,LL3,CD1,CD2,CD4,CD3,CK,DK,TJ1,TJ,Kri,Se1,Se,L1,L3} for both elliptic and parabolic integro-differential equations. In this paper, we follow \cite{SS} to assume a much weaker uniform ellipticity. Roughly speaking, we let $I$ be uniformly elliptic with respect to a larger class of linear nonlocal operators where the kernels $K$ satisfy the right hand side of $(\ref{eq:uell1})$ in an integral sense and the left hand side of that in a symmetric subset of each annulus domain with positive measure. The main tool we use is the weak Harnack inequality obtained in \cite{SS}. With the weak Harnack inequality, we are able to prove the oscillation between the upper and lower semicontinuous envelope of the discontinuous viscosity solution $u$ in the ball $B_r$ is of order $r^\alpha$ for some $\alpha>0$ and any small $r>0$. This proves that $u$ is H\"older continuous and thus it is a viscosity solution of $(\ref{eq:gennon})$. Recently, L. Silvestre applied the regularity for nonlocal equations under this weak ellipticity to obtain the regularity for the homogeneous Boltzmann equation without cut-off, see \cite{L2}. We also want to mention that M. Kassmann, M. Rang and R. Schwab studied H\"older regularity for a class of integro-differential operators with kernels which are positive along some given rays or cone-like sets, see \cite{KRS}.

To complete the existence results, we construct continuous sub/supersolutions in both uniformly elliptic and degenerate cases in Section 5. In the uniformly elliptic case, we follow the idea of \cite{XJ1} to construct appropriate barrier functions. We then use them to construct a subsolution and a supersolution which satisfy the boundary condition. The weak uniform ellipticity and the lower order terms of $I$ make the proofs more involved. With all these ingredients in hand, we can conclude one of the main results in this manuscript that $(\ref{eq:gennon})$ admits a viscosity solution if $I$ is uniformly elliptic, see Theorem $\ref{mainthm1}$ in Section 5.1. This main result generalizes nearly all the previous existence results for uniformly elliptic integro-differential equations. In the degenerate case,  it is natural to construct a sub/supersolution only for $(\ref{eq:belisa})$ since we have little information about the nonlocal operator $I$. Moreover, we need to assume the nonlocal Bellman-Isaacs equation in $(\ref{eq:belisa})$ satisfies the coercive assumption, i.e., $c_{ab}\geq\gamma$ for some $\gamma>0$. The coercive assumption is often made to study uniqueness, existence and regularity of viscosity solutions of degenerate elliptic PDEs and integro-PDEs, see \cite{BCI0,BI,CIL,Is,Is1,IL,jk1,M,MS}. In Section 5.2, we obtain a subsolution and a supersolution which satisfy the boundary condition in the degenerate case. The difficulty here lies in giving a degenerate assumption on the kernels which allows us to construct barrier functions. Roughly speaking, we only need to assume that the kernels $K_{ab}(x,\cdot)$ are non-degenerate in the outer-pointing normal direction of the boundary for the points $x$ which are sufficiently close to the boundary. That means we allow our kernels $K_{ab}$ to be degenerate in the whole domain. Then we can conclude the second main result, the existence of a discontinuous viscosity solution of $(\ref{eq:belisa})$, given in Theorem $\ref{thm:disvis}$. If the comparison principle holds for $(\ref{eq:belisa})$, we obtain the discontinuous viscosity solution is a viscosity solution. In the end, we want to notice that our method could be adapted to the nonlocal parabolic equations for obtaining the corresponding existence results.

\section{Notation and definitions}

We write $B_\delta$ for the open ball centered at the origin with radius $\delta>0$ and $B_{\delta}(x):=B_\delta+x$. We set $\Omega_{\delta}:=\{x\in\Omega; {\rm dist}(x,\partial\Omega)>\delta\}$ for $\delta>0$. For each
non-negative integer $r$ and $0<\alpha\le1$, we denote by
$C^{r,\alpha}(\Omega)$ ($C^{r,\alpha}(\bar\Omega)$) the subspace of
$C^{r,0}(\Omega)$ ($C^{r,0}(\bar\Omega)$) consisting functions whose
$r$th partial derivatives are locally (uniformly) $\alpha$-H\"older
continuous in $\Omega$.
For any $u\in
C^{r,\alpha}(\bar\Omega)$, where $r$ is a non-negative integer and
$0\le\alpha\le 1$, define
\[
[u]_{r, \alpha; \Omega}:=\left\{\begin{array}{ll}
\sup_{x\in\Omega,|j|=r}|\partial^{j}u(x)|,&\hbox{if}\, \alpha=0;\\
\sup_{x, y\in \Omega, x\not
=y,|j|=r}\frac{|\partial^{j}u(x)-\partial^{j}u(y)|}{|x-y|^{\alpha}},
&\hbox{if}\, \alpha>0,\end{array}\right.
\]
and
\[
\|u\|_{C^{r, \alpha}(\bar \Omega)}=\left\{\begin{array}{ll}
\sum_{j=0}^{r}[u]_{j, 0, \Omega}, &\hbox{if}\, \alpha=0;\\
\|u\|_{C^{r,0}(\bar \Omega)}+[u]_{r, \alpha; \Omega}, &\hbox{if}\,
\alpha>0.\end{array}\right.
\]
For simplicity, we use the notation
$C^\beta(\Omega)$ ($C^{\beta}(\bar\Omega)$), where $\beta>0$, to
denote the space $C^{r,\alpha}(\Omega)$
($C^{r,\alpha}(\bar\Omega)$), where $r$ is the largest integer
smaller than $\beta$ and $\alpha=\beta-r$. The set $C_b^{\beta}(\Omega)$ consist of functions from $C^{\beta}(\Omega)$ which are bounded. We write $USC(\mathbb R^n)$ ($LSC(\mathbb R^n)$) for the space of upper (lower) semicontinuous function in $\mathbb R^n$. 

%We write $u\in L^1(\mathbb R^n,\frac{1}{1+|z|^{n+\sigma}})$ for any $0<\sigma<2$ to denote
%\begin{equation*}
%\int_{\mathbb R^n}\frac{|u(z)|}{1+|z|^{n+\sigma}}dz<\infty.
%\end{equation*}

%We let $BUC(\mathbb R^n)$ denote the space of bounded and uniformly continuous functions in $\mathbb R^n$.

We will give a definition of viscosity solutions of $(\ref{eq:gennon})$. We first state the general assumptions on the nonlocal operator $I$ in $(\ref{eq:gennon})$. For any $\delta>0$, $r,s\in\mathbb R$, $x,x_k\in\Omega$, $\varphi,\varphi_k,\psi\in C^{2}(B_\delta(x))\cap L^{\infty}(\mathbb R^n)$, we assume:
\begin{itemize}
\item[({\rm A0})] The function $(x,r)\to I(x,r,\varphi(\cdot))$ is continuous in $B_\delta(x)\times\mathbb R$.
\item[({\rm A1})] If $x_k\to x$ in $\Omega$, $\varphi_k\to \varphi$ a.e. in $\mathbb R^n$, $\varphi_k\to \varphi$ in $C^{2}(B_{\delta}(x))$ and $\{\varphi_k\}_k$ is uniformly bounded in $\mathbb R^n$, then
\begin{equation*}
I(x_k,r,\varphi_k(\cdot))\to I(x,r,\varphi(\cdot)).
\end{equation*}
\item[({\rm A2})] If $r\leq s$, then $I(x,r,\varphi(\cdot))\leq I(x,s,\varphi(\cdot))$.
\item[({\rm A3})] For any constant $C$, $I(x,r,\varphi(\cdot)+C)=I(x,r,\varphi(\cdot))$.
\item[({\rm A4})] If $\varphi$ touches $\psi$ from above at $x$, then $I(x,r,\varphi(\cdot))\leq I(x,r,\psi(\cdot))$.
\end{itemize}
\begin{remark}
If $I$ is uniformly elliptic and satisfies $({\rm A0})$, $({\rm A2})$, then $({\rm A0})$-$({\rm A4})$ hold for $I$. See Lemma $\ref{lem:equiv}$.
\end{remark}
\begin{remark}\label{rem:form}
The nonlocal operator $I$ in \cite{SS} has only two components, i.e., $(x,\varphi)\to I(x,\varphi(\cdot))$. Here we let our nonlocal operator $I$ have three components and assume $({\rm A2})$-$({\rm A3})$ hold. It is because that we want to let $I$ include the left hand side of the nonlocal Bellman-Isaacs equation in $(\ref{eq:belisa})$ and, moreover, want to describe the following two properties 
\begin{equation*}
-I_{ab}[x,\varphi+C]+b_{ab}(x)\cdot \nabla(\varphi+C)(x)=-I_{ab}[x,\varphi]+b_{ab}(x)\cdot \nabla\varphi(x),
\end{equation*}
\begin{equation*}
c_{ab}(x)r\leq c_{ab}(x)s\quad\text{if $r\leq s$}
\end{equation*}
in abstract forms.
\end{remark}
\begin{remark}
The left hand side of the nonlocal Bellman-Isaacs equation in $(\ref{eq:belisa})$ satisfies {\rm (A0)}-{\rm (A4)} if $(\ref{eq:levmea})$ holds and its coefficients $K_{ab}$, $b_{ab}$, $c_{ab}$ and $f_{ab}$ are uniformly continuous with respect to $x$ in $\Omega$, uniformly in $a\in\mathcal{A}$, $b\in \mathcal{B}$. See \cite{ns} for when the nonlocal operator $I$ has a min-max structure. 
\end{remark}
Throughout the paper, we always assume the nonlocal operator $I$ satisfies {\rm (A0)}-{\rm (A4)}.

\begin{definition}\label{de:vis1}
A bounded function $u\in USC(\mathbb R^n)$ is a viscosity subsolution of $I=0$ in $\Omega$ if whenever $u-\varphi$ has a maximum over $\mathbb R^n$ at $x\in\Omega$ for $\varphi\in C_b^2(\mathbb R^n)$, then
\begin{equation*}
I\left(x,u(x),\varphi(\cdot)\right)\leq 0.
\end{equation*}
A bounded function $u\in LSC(\mathbb R^n)$ is a viscosity supersolution of $I=0$ in $\Omega$ if whenever $u-\varphi$ has a minimum over $\mathbb R^n$ at $x\in\Omega$ for $\varphi\in C_b^2(\mathbb R^n)$, then
\begin{equation*}
I\left(x,u(x),\varphi(\cdot)\right)\geq 0.
\end{equation*}
A bounded function $u$ is a viscosity solution of $I=0$ in $\Omega$ if it is both a viscosity subsolution and viscosity supersolution
of $I=0$ in $\Omega$.
\end{definition}
\begin{remark}
In Definition $\ref{de:vis1}$, all the maximums and minimums can be replaced by strict maximums and minimums.
\end{remark}
\begin{definition}\label{def:visbou}
A bounded function $u$ is a viscosity subsolution of $(\ref{eq:gennon})$ if $u$ is a viscosity subsolution of $I=0$ in $\Omega$ and $u\leq g$ in $\Omega^c$. 
A bounded function $u$ is a viscosity supersolution of $(\ref{eq:gennon})$ if $u$ is a viscosity supersolution of $I=0$ in $\Omega$ and $u\geq g$ in $\Omega^c$. A bounded function $u$ is a viscosity solution of $(\ref{eq:gennon})$ if $u$ is a viscosity subsolution and supersolution of $(\ref{eq:gennon})$.
\end{definition}
We will use the following notations: if $u$ is a function on $\Omega$, then, for any $x\in\Omega$,
\begin{equation*}
u^*(x)=\lim_{r\to 0}\sup\{u(y); y\in\Omega\,\,\text{and}\,\,|y-x|\leq r\},
\end{equation*}
\begin{equation*}
u_*(x)=\lim_{r\to 0}\inf\{u(y); y\in\Omega\,\,\text{and}\,\,|y-x|\leq r\}.
\end{equation*}
One calls $u^*$ the upper semicontinuous envelope of $u$ and $u_*$ the lower semicontinuous envelope of $u$.

We then give a definition of discontinuous viscosity solutions of $(\ref{eq:gennon})$.
\begin{definition}\label{de:disvis}
A bounded function $u$ is a discontinuous viscosity subsolution of $(\ref{eq:gennon})$ if $u^*$ is a viscosity subsolution of $(\ref{eq:gennon})$.
A bounded function $u$ is a discontinuous viscosity supersolution of $(\ref{eq:gennon})$ if $u_*$ is a viscosity supersolution of  $(\ref{eq:gennon})$.
A function $u$ is a discontinuous viscosity solution of $(\ref{eq:gennon})$ if it is both a discontinuous viscosity subsolution and a discontinuous viscosity supersolution of $(\ref{eq:gennon})$.
\end{definition}
\begin{remark}
If $u$ is a discontinuous viscosity solution of $(\ref{eq:gennon})$ and $u$ is continuous in $\mathbb R^n$, then $u$ is a viscosity solution of $(\ref{eq:gennon})$.
\end{remark}

\section{Perron's method}\label{sec:per}

In this section, we obtain existence of a discontinuous viscosity solution of $(\ref{eq:gennon})$ by Perron's method. We remind you that $I$ satisfies {\rm (A0)}-{\rm (A4)}.
\begin{lemma}\label{lem:sta}
Let $\mathcal{F}$ be a family of viscosity subsolutions of $I=0$ in $\Omega$. Let $w(x)=\sup\{u(x):u\in\mathcal{F}\}$ in $\mathbb R^n$and assume that $w^*(x)<\infty$ for all $x\in\mathbb R^n$. Then $w$ is a discontinuous viscosity subsolution of $I=0$ in $\Omega$.
\end{lemma}
\begin{proof}
Suppose that $\varphi$ is a $C_b^2(\mathbb R^n)$ function such that $w^*-\varphi$ has a strict maximum (equal $0$) at $x_0\in\Omega$ over $\mathbb R^n$. We can construct a uniformly bounded sequence of $C^2(\mathbb R^n)$ functions $\{\varphi_m\}_m$ such that $\varphi_m=\varphi$ in $B_1(x_0)$, $\varphi\leq\varphi_m$ in $\mathbb R^n$, $\sup_{x\in B_2^c(x_0)}\{w^*(x)-\varphi_m(x)\}\leq-\frac{1}{m}$ and $\varphi_m\to\varphi$ pointwise. Thus, for any positive integer $m$, $w^*-\varphi_m$ has a strict maximum (equal $0$) at $x_0$ over $\mathbb R^n$. Therefore, $\sup_{x\in B_1^c(x_0)}\{w^*(x)-\varphi_m(x)\}=\epsilon_m<0$. By the definition of $w^*$, we have, for any $u\in\mathcal{F}$, $\sup_{x\in B_1^c(x_0)}\{u(x)-\varphi_m(x)\}\leq\epsilon_m<0$. Again, by the definition of $w^*$, we have, for any $\epsilon_m<\epsilon<0$, there exist $u_\epsilon\in\mathcal{F}$ and $\bar x_\epsilon\in B_1(x_0)$ such that $u_{\epsilon}(\bar x_\epsilon)-\varphi(\bar x_\epsilon)>\epsilon$. Since $u_\epsilon\in USC(\mathbb R^n)$ and $\varphi_m\in C_b^2(\mathbb R^n)$, there exists $x_\epsilon\in B_1(x_0)$ such that $u_{\epsilon}(x_\epsilon)-\varphi_m(x_\epsilon)=\sup_{x\in\mathbb R^n}\{u_\epsilon(x)-\varphi(x)\}\geq u_{\epsilon}(\bar x_\epsilon)-\varphi_m(\bar x_\epsilon)>\epsilon$. Since $w^*-\varphi_m$ attains a strict maximum (equal $0$) at $x_0$ over $\mathbb R^n$ and $u\leq w^*$ for any $u\in\mathcal{F}$, then $u_\epsilon(x_\epsilon)\to w^*(x_0)$ and $x_\epsilon\to x_0$ as $\epsilon\to 0^-$. Since $u_\epsilon$ is a viscosity subsolution of $I=0$ in $\Omega$, we have 
\begin{equation}\label{eq3.1}
I(x_\epsilon,u_\epsilon(x_\epsilon),\varphi_m(\cdot))\leq 0.
\end{equation}
Since $x_\epsilon\to x_0$, $u_\epsilon(x_\epsilon)\to w^*(x_0)$ as $\epsilon\to 0^-$, $\varphi_m=\varphi$ in $B_1(x_0)$, $\varphi_m\to \varphi$ pointwise, $\{\varphi_m\}_m$ is uniformly bounded, $\varphi\in C_b^2(\mathbb R^n)$, {\rm (A0)} and {\rm (A1)} hold, we have, letting $\epsilon\to 0^-$ and $m\to+\infty$ in $(\ref{eq3.1})$,
\begin{equation*}
I(x_0,w^*(x_0),\varphi(\cdot))\leq 0.
\end{equation*}
Therefore, $w$ is a discontinuous viscosity subsolution of $I=0$.
\end{proof}

\begin{theorem}\label{thm:per}
Let $\underbar u, \bar u$ be bounded continuous functions and be respectively a viscosity subsolution and a viscosity supersolution of $I=0$ in $\Omega$. Assume moreover that $\bar u=\underbar u=g$ in $\Omega^c$ for some bounded continuous function $g$ and $\underbar u\leq \bar u$ in $\mathbb R^n$. Then
\begin{equation*}
w(x)=\sup_{u\in\mathcal{F}}u(x),
\end{equation*}
where $\mathcal{F}=\{u\in C^0(\mathbb R^n);\,\,\underbar u\leq u\leq \bar u\,\, in\,\,\mathbb R^n\,\, and \,\, u\,\,\text{is a viscosity subsolution of $I=0$ in $\Omega$}\}$,
 is a discontinuous viscosity solution of $(\ref{eq:gennon})$.
\end{theorem}
\begin{proof}
Since $\underbar u\in\mathcal{F}$, then $\mathcal{F}\not=\emptyset$. Thus, $w$ is well defined, $\underbar u\leq w\leq\bar u$ in $\mathbb R^n$ and $w=\bar u=\underbar u$ in $\Omega^c$. By Lemma $\ref{lem:sta}$, $w$ is a discontinuous viscosity subsolution of $G=0$ in $\Omega$. We claim that $w$ is a discontinuous viscosity supersolution of $G=0$ in $\Omega$. If not, there exist a point $ x_0\in\Omega$ and a function $\varphi\in C_b^2(\mathbb R^n)$ such that $w_*-\varphi$ has a strict minimum (equal $0$) at the point $x_0$ over $\mathbb R^n$ and 
\begin{equation*}
I(x_0,w_*(x_0),\varphi(\cdot))<-\epsilon_0,
\end{equation*}
where $\epsilon_0$ is a positive constant. Thus, we can find sufficiently small constants $\epsilon_1>0$ and $\delta_0>0$ such that $B_{\delta_0}(x_0)\subset \Omega$ and there exists a $C_b^2(\mathbb R^n)$ function $\varphi_{\epsilon_1}$ satisfying that $\varphi_{\epsilon_1}=\varphi$ in $B_{\delta_0}(x_0)$, $\varphi_{\epsilon_1}\leq\varphi$ in $\mathbb R^n$, $\inf_{x\in B_{2\delta_0}^c(x_0)}\{w_*(x)-\varphi_{\epsilon_1}(x)\}\geq \epsilon_1>0$ and
\begin{equation}\label{eq3.3}
I(x_0,\varphi_{\epsilon_1}(x_0),\varphi_{\epsilon_1}(\cdot))<-\frac{\epsilon_0}{2}.
\end{equation}
Thus, by {\rm (A0)}, there exists $\delta_1<\delta_0$ such that, for any $x\in B_{\delta_1}(x_0)$,
\begin{equation}\label{eq3.4}
I(x,\varphi_{\epsilon_1}(x),\varphi_{\epsilon_1}(\cdot))<-\frac{\epsilon_0}{4}.
\end{equation}
By the definition of $w$, we have $\varphi_{\epsilon_1}\leq w_*\leq \bar u$ in $\mathbb R^n$. If $\varphi_{\epsilon_1}(x_0)=w_*(x_0)=\bar u(x_0)$, then $\bar u-\varphi_{\epsilon_1}$ has a strict minimum at point $x_0$ over $\mathbb R^n$. Since $\bar u$ is a viscosity supersolution of $I=0$ in $\Omega$, we have
\begin{equation*}
I(x_0,\varphi_{\epsilon_1}(x_0),\varphi_{\epsilon_1}(\cdot))\geq 0,
\end{equation*}
which contradicts with $(\ref{eq3.3})$. Thus, we have $\varphi_{\epsilon_1}(x_0)<\bar u(x_0)$. Since $\bar u$ and $\varphi_{\epsilon_1}$ are continuous functions in $\mathbb R^n$, we have $\varphi_{\epsilon_1}(x)<\bar u(x)-\epsilon_2$ in $B_{\delta_2}(x_0)$ for some $0<\delta_2<\delta_1$ and $\epsilon_2>0$. We define
\begin{equation*}
\Delta_r=\sup_{x\in B_r^c(x_0)}\{\varphi_{\epsilon_1}(x)-w_*(x)\}.
\end{equation*}
Since $\inf_{x\in B_{2\delta_0}^c(x_0)}\{w_*(x)-\varphi_{\epsilon_1}(x)\}\geq \epsilon_1>0$, $w_*-\varphi_{\epsilon_1}$ has a strict minimum (equal $0$) at the point $x_0$ and $-w_*\in USC(\mathbb R^n)$, we have $\Delta_r<0$ for each $r>0$. For any $y\in \bar\Omega\setminus B_r(x_0)$, there exists a function $v_y\in\mathcal{F}$ such that $v_y(y)-\varphi_{\epsilon_1}(y)\geq -\frac{3\Delta_r}{4}$. Since $v_y$ and $\varphi_{\epsilon_1}$ are continuous in $\mathbb R^n$, there exists a positive constant $\delta_y$ such that $\inf_{x\in B_{\delta_y}(y)}\{v_y(x)-\varphi_{\epsilon_1}(x)\}\geq-\frac{\Delta_r}{2}$. Since $\bar\Omega\setminus B_r(x_0)$ is a compact set in $\mathbb R^n$, there exists a finite set $\{y_i\}_{i=1}^{n_r}\subset \bar\Omega\setminus B_r(x_0)$ such that $\bar\Omega\setminus B_r(x_0)\subset \cup_{i=1}^{n_r} B_{\delta_{y_i}}(y_i)$. Thus, we define 
\begin{equation*}
v_r(x)=\sup_{1\leq i\leq n_r}\{v_{y_i}(x)\},\quad x\in\mathbb R^n.
\end{equation*}
By Lemma $\ref{lem:sta}$ and the definition of $v_r$, we have $v_r\in\mathcal{F}$ and $\inf_{x\in\bar\Omega\setminus B_r(x_0)}\{v_r(x)-\varphi_{\epsilon_1}(x)\}\geq -\frac{\Delta_r}{2}$. Let $\alpha_r$ be a constant such that $0<\alpha_r<\frac{1}{2}$ and $-\alpha_r\Delta_r<\epsilon_2$.
Thus, we define 
\begin{equation*}
U(x)=\left\{\begin{array}{ll} \max\{\varphi_{\epsilon_1}(x)-\alpha\Delta_r,\,\,v_r(x)\},
\quad x\in B_r(x_0),\\
\qquad\qquad\qquad\qquad\quad\,\,\, v_r(x),\quad x\in B_r^c(x_0),
\end{array}
  \right.
\end{equation*}
where $0<r<\delta_2$ and $0<\alpha<\alpha_r$. By the definition of $U$, we obtain $U\in C^0(\mathbb R^n)$, $\underbar u\leq U\leq\bar u$ in $\mathbb R^n$, and there exists a squence $\{x_n\}_n\subset B_r(x_0)$ such that $x_n\to x_0$ as $n\to+\infty$ and $U(x_n)>w(x_n)$.

We claim that $U$ is a viscosity subsolution of $I=0$ in $\Omega$. For any $y\in\Omega$, suppose that there is a function $\psi\in C_b^2(\mathbb R^n)$ such that $U-\psi$ has a maximum (equal $0$) at $y$ over $\mathbb R^n$. We then divide the proof into two cases.

Case 1: $U(y)=v_r(y)$.

Since $v_r\leq U\leq \psi$ in $\mathbb R^n$, then $v_r-\psi$ has a maximum (equal $0$) at $y$ over $\mathbb R^n$. We recall that $v_r$ is a viscosity subsolution of $I=0$ in $\Omega$. Therefore, we have 
\begin{equation*}
I(y,U(y),\psi(\cdot))\leq 0.
\end{equation*}

Case 2: $U(y)=\varphi_{\epsilon_1}(y)-\alpha\Delta_r$.

We first notice that $y\in B_r(x_0)$. Since $\varphi_{\epsilon_1}-\alpha\Delta_r\leq U\leq \psi$ in $B_r(x_0)$, then $\varphi_{\epsilon_1}-\alpha\Delta_r-\psi\leq 0$ in $B_r(x_0)$. By the definition of $U$, we have $\psi\geq U=v_r$ in $B_r^c(x_0)$. Thus, $\varphi_{\epsilon_1}-\alpha\Delta_r-\psi\leq \varphi_{\epsilon_1}-\alpha\Delta_r-v_r\leq \frac{\Delta_r}{2}-\alpha\Delta_r\leq 0$ in $B_r^c(x_0)$. Therefore, we have $\varphi_{\epsilon_1}-\alpha\Delta_r-\psi$ has a maximum (equal $0$) at $y\in B_r(x_0)\subset B_{\delta_1}(x_0)$ over $\mathbb R^n$. Since $(\ref{eq3.4})$, {\rm (A0)}, {\rm (A3)}-{\rm (A4)} hold, we can choose sufficiently small $\alpha$ independent of $\psi$ such that
\begin{equation*}
I(y,\psi(y),\psi(\cdot))\leq I(y,\varphi_{\epsilon_1}(y)-\alpha\Delta_r,\varphi_{\epsilon_1}(\cdot))\leq 0.
\end{equation*}

Based on the two cases, we have that $U$ is a viscosity subsolution of $I=0$ in $\Omega$. Therefore, $U\in\mathcal{F}$, which contradicts with the definition of $w$. Thus, $w$ is a discontinuous viscosity supersolution of $I=0$ in $\Omega$. Therefore, $w$ is a discontinuous viscosity solution of $I=0$ in $\Omega$. Since $w=g$ in $\Omega^c$, then $w$ is a discontinuous viscosity solution of $(\ref{eq:gennon})$.
\end{proof}

\begin{remark}
Under the assumptions of Theorem $\ref{thm:per}$, if the comparison principle holds for $(\ref{eq:gennon})$, the discontinuous viscosity solution $w$ is the unique viscosity solution of $(\ref{eq:gennon})$. For example, if $I$ is a translation invariant nonlocal operator, $(\ref{eq:gennon})$ admits a unique viscosity solution.
\end{remark}

Before applying Theorem $\ref{thm:per}$ to $(\ref{eq:belisa})$, we now give the precise assumptions on its equation. For any $0<\lambda\leq\Lambda$ and $0<\sigma<2$, we consider the family of kernels $K:\mathbb R^n\to\mathbb R$ satisfying the following assumptions.\\
{\rm (H0)} $K(z)\geq 0$ for any $z\in\mathbb R^n$.\\
{\rm (H1)} For any $\delta>0$,
\begin{equation*}
\int_{B_{2\delta}\setminus B_\delta}K(z)dz\leq (2-\sigma)\Lambda \delta^{-\sigma}.
\end{equation*}
{\rm (H2)} For any $\delta>0$,
\begin{equation*}
\left|\int_{B_{2\delta}\setminus B_\delta}zK(z)dz\right|\leq \Lambda|1-\sigma|\delta^{1-\sigma}.
\end{equation*}
We define our nonlocal operator
\begin{equation}\label{eq:nonope1}
I_{ab}[x,u]:=\int_{\mathbb R^n}\delta_z u(x)K_{ab}(x,z)dz,
\end{equation}
where
\begin{equation*}
\delta_zu(x):=\left\{\begin{array}{ll} u(x+z)-u(x),\,\qquad\qquad\qquad\qquad\qquad\text{if $\sigma<1$},\\
u(x+z)-u(x)-\mathbbm 1_{B_1}(z)\nabla u(x)\cdot z,\qquad \text{if $\sigma=1$},\\
u(x+z)-u(x)-\nabla u(x)\cdot z,\qquad\qquad\quad \text{if $\sigma>1$}.\\
\end{array}
  \right.
\end{equation*}
We consider the following nonlocal Bellman-Isaacs equation
\begin{equation}\label{eq3.5}
\sup_{a\in\mathcal{A}}\inf_{b\in\mathcal{B}}\{-I_{ab}[x,u]+b_{ab}(x)\cdot \nabla u(x)+c_{ab}(x)u(x)+f_{ab}(x)\}=0,\quad\text{in $\Omega$.}
\end{equation}
\begin{corollary}\label{col:biper}
Assume that $0<\sigma<2$, $b_{ab}\equiv 0$ in $\Omega$ if $\sigma<1$ and $c_{ab}\geq 0$ in $\Omega$. Let $\underbar u$, $\bar u$ be bounded continuous functions and be respectively a viscosity subsolution and a viscosity supersolution of $(\ref{eq3.5})$ where $\{K_{ab}(\cdot,z)\}_{a,b,z}$, $\{b_{ab}\}_{a,b}$, $\{c_{ab}\}_{a,b}$ and $\{f_{ab}\}_{a,b}$ are sets of uniformly continuous functions in $\Omega$, uniformly in $a\in\mathcal{A}$, $b\in\mathcal{B}$, and $\{K_{ab}(x,\cdot):x\in\Omega, a\in\mathcal{A}, b\in\mathcal{B}\}$ are kernels satisfying {\rm (H0)}-{\rm (H2)}.
Assume moreover that $\bar u=\underbar u=g$ in $\Omega^c$ for some bounded continuous function $g$ and $\underbar u\leq\bar u$ in $\mathbb R^n$. Then
\begin{equation*}
w(x)=\sup_{u\in\mathcal{F}}u(x),
\end{equation*}
where $\mathcal{F}=\{u\in C^0(\mathbb R^n);\,\,\underbar u\leq u\leq \bar u\,\, in\,\,\mathbb R^n\,\, and \,\, u\,\,\text{is a viscosity subsolution of $(\ref{eq3.5})$}\}$,
 is a discontinuous viscosity solution of $(\ref{eq:belisa})$.
\end{corollary}
\begin{proof}
We define 
\begin{equation*}
I(x,r,u(\cdot)):=\sup_{a\in\mathcal{A}}\inf_{b\in\mathcal{B}}\{-I_{ab}[x,u]+b_{ab}(x)\cdot \nabla u(x)+c_{ab}(x)r+f_{ab}(x)\}.
\end{equation*}
It follows from {\rm (H1)} and {\rm (H2)} that $I_{ab}$ satisfies $(\ref{eq:levmea})$, see Lemma 2.3 in \cite{SS}. Then, by $(\ref{eq:levmea})$ and uniform continuity of the coefficients, {\rm (A0)} and {\rm (A1)} hold. Since $c_{ab}\geq 0$ in $\Omega$, {\rm (A2)} holds. By {\rm (H0)} and the structure of $I_{ab}$, {\rm (A3)} and {\rm (A4)} hold. 
\end{proof}

\section{H\"older estimates}

In this section we give H\"older estimates of the discontinuous viscosity solution constructed by Perron's method in the above section. To obtain H\"older estimates, we will assume that the nonlocal operator $I$ is uniformly elliptic. 

We define $\mathcal{L}:=\mathcal{L}(\sigma,\lambda,\Lambda)$ is the class of all the nonlocal operators of form
\begin{equation*}
Lu(x):=\int_{\mathbb R^n}\delta_zu(x)K(z)dz,
\end{equation*}
where $K$ is a kernel satisfying the assumptions {\rm (H0)}-{\rm (H2)} given above and\\
{\rm (H3)} There exist positive constants $\lambda$ and $\mu$ such that, for any $\delta>0$, there is a set $A_\delta$ satisfying
\begin{itemize}
\item[(i)] $A_\delta\subset B_{2\delta}\setminus B_\delta$;
\item[(ii)] $A_\delta=-A_\delta$;
\item[(iii)] $|A_\delta|\geq\mu|B_{2\delta}\setminus B_\delta|$;
\item[(iv)] $K(z)\geq (2-\sigma)\lambda \delta^{-n-\sigma}$ for any $z\in A_\delta$.
\end{itemize}
We note that we will also write $K\in\mathcal{L}$ if the corresponding nonlocal operator $L\in\mathcal{L}$. We then define the extremal operators \begin{equation*}
M_{\mathcal{L}}^+u(x):=\sup_{L\in\mathcal{L}}Lu(x),
\end{equation*}
\begin{equation*}
 M_{\mathcal{L}}^-u(x):=\inf_{L\in\mathcal{L}}Lu(x).
 \end{equation*}
We denote by $m:[0,+\infty)\to[0,+\infty)$ a modulus of continuity. We say that the nonlocal operator $I$ is uniformly elliptic if for every $r,s\in\mathbb R$, $x\in\Omega$, $\delta>0$, $\varphi,\psi\in C^{2}(B_\delta(x))\cap L^{\infty}(\mathbb R^n)$, 
\begin{eqnarray*}
&&M_{\mathcal L}^-(\varphi-\psi)(x)-C_0|\nabla (\psi-\varphi)(x)|-m(|r-s|)\\
&\leq&  I(x,r,\psi(\cdot))-I(x,s,\varphi(\cdot))\\
&\leq& M_{\mathcal{L}}^+(\varphi-\psi)(x)+C_0|\nabla(\psi-\varphi)(x)|+m(|r-s|),
\end{eqnarray*}
where $C_0$ is a non-negative constant such that $C_0=0$ if $\sigma<1$.
\begin{remark}
The definition of uniform ellipticity is different from that in \cite{SS} since the nonlocal operator $I$ contains the second component $r$.
\end{remark}
\begin{lemma}\label{lem:equiv}
If the nonlocal operator $I$ is uniformly elliptic and satisfies {\rm (A0)}, {\rm (A2)}, then $I$ satisfies {\rm (A0)}-{\rm (A4)}.
\end{lemma}
\begin{proof}
Suppose that $\delta>0$, $x_k\to x$ in $\Omega$, $\varphi_k\to \varphi$ a.e. in $\mathbb R^n$, $\varphi_k\to\varphi$ in $C^{2}(B_\delta(x))$ and $\{\varphi_k\}_k$ is uniformly bounded in $\mathbb R^n$. Since $I$ is uniformly elliptic, we have, for any $r\in\mathbb R$,
\begin{eqnarray}
&&M_{\mathcal L}^-(\varphi-\varphi_k)(x_k)-C_0|\nabla (\varphi_k-\varphi)(x_k)|\nonumber\\
&\leq&I(x_k,r,\varphi_k(\cdot))-I(x_k,r,\varphi(\cdot))\nonumber\\
&\leq&M_{\mathcal{L}}^+(\varphi-\varphi_k)(x_k)+C_0|\nabla(\varphi_k-\varphi)(x_k)|.\label{eq:4.1}
\end{eqnarray}
Since $K\in\mathcal{L}$, then, by Lemma 2.3 in \cite{SS}, $K$ satisfies $(\ref{eq:levmea})$. Letting $k\to+\infty$ in $(\ref{eq:4.1})$, we have, by {\rm(A0)},
\begin{equation*}
\lim_{k\to+\infty}I(x_k,r,\varphi_k(\cdot))=I(x,r,\varphi(\cdot)).
\end{equation*}
Therefore, {\rm(A1)} holds. For any constant $C$, we have
\begin{equation*}
0=M_{\mathcal L}^-(-C)-C_0|\nabla C|\leq I(x,r,\varphi(\cdot)+C)-I(x,r,\varphi(\cdot))\leq M_{\mathcal{L}}^+(-C)+C_0|\nabla C|=0.
\end{equation*}
Thus, {\rm(A3)} holds. If $\varphi$ touches a $C^{2}(B_\delta(x))\cap L^{\infty}(\mathbb R^n)$ function $\psi$ from above at $x$, then
\begin{equation*}
I(x,r,\varphi)-I(x,r,\psi)\leq M_{\mathcal{L}}^+(\psi-\varphi)(x)\leq 0.
\end{equation*}
Therefore, {\rm(A4)} holds.
\end{proof}

%For any $0<\sigma<2$ and $0<\lambda\leq \Lambda$, a class $\mathcal{L}_0(\sigma,\lambda,\Lambda)$ is a set of linear operators $I$ of the form
%\begin{equation}
%Iu(x)=\int_{\mathbb R^n}[u(x+z)-u(x)-\mathbbm{1}_{B_1(0)}(z)Du(x)\cdot z]k(z)dz,
%\end{equation}
%where the kernel $k$ is symmetric, i.e.,
%\begin{equation*}
%k(z)=k(-z),
%\end{equation*} 
%and satisfies for all $z\in\mathbb R^n\setminus \{0\}$
%\begin{equation*}
%(2-\sigma)\frac{\lambda}{|z|^{n+\sigma}}\leq k(z)\leq (2-\sigma)\frac{\Lambda}{|z|^{n+\sigma}}.
%\end{equation*}
%Since the kernel $k(z)$ is symmetric, we also have
%\begin{equation*}
%Iu=\int_{\mathbb R^n}\delta u(x,z)k(z)dz,
%\end{equation*}
%where $\delta(x,z)=u(x+z)+u(x-z)-2u(x)$.

The following lemma is an elliptic version of Theorem 6.1 in \cite{SS}.
\begin{lemma}\label{lem:har}
Assume $0<\sigma_0\leq\sigma<2$, $C_0,C_1\geq 0$, and further assume $C_0=0$ if $\sigma<1$. Let $u$ be a viscosity supersolution of
\begin{equation*}
M_{\mathcal{L}}^-u-C_0|\nabla u|= C_1\quad \text{in $B_2$}
\end{equation*}
and $u\geq 0$ in $\mathbb R^n$. Then there exist constants $C$ and $\epsilon_3$ such that
\begin{equation*}
\left(\int_{B_{1}}u^{\epsilon_3} dx\right)^{\frac{1}{\epsilon_3}}\leq C(\inf_{B_{1}}u+C_1),
\end{equation*}
where $\epsilon_3$ and $C$ depend on $\sigma_0$, $\lambda$, $\Lambda$, $C_0$, $n$ and $\mu$.
\end{lemma}
The following Lemma is a direct Corollary of Lemma 4.2.
\begin{corollary}\label{cor:har}
Assume $0<\sigma_0\leq\sigma<2$, $0<r<1$, $C_0,C_1\geq 0$, and further assume $C_0=0$ if $\sigma<1$. Let $u$ be a viscosity supersolution of
\begin{equation*}
M_{\mathcal{L}}^-u-C_0|\nabla u|= C_1\quad \text{in $B_{2r}$}
\end{equation*}
and $u\geq 0$ in $\mathbb R^n$. Then there exist constants $C$ and $\epsilon_3$ such that
\begin{equation}\label{eq:har}
(|\{u>t\}\cap B_r|)\leq Cr^n(u(0)+C_1r^\sigma)^{\epsilon_3} t^{-\epsilon_3},\quad\text{for any $t\geq0$},
\end{equation}
where $\epsilon_3$ and $C$ depend on $\sigma_0$, $\lambda$, $\Lambda$, $C_0$, $n$ and $\mu$.
\end{corollary}
\begin{proof}
Now let $v(x)=u(rx)$. By Lemma 2.2 in \cite{SS}, we have
\begin{equation}\label{eq:4.3}
M_{\mathcal{L}}^-v-C_0r^{\sigma-1}|\nabla v|\leq C_1r^{\sigma},\quad\text{in $B_2$}.
\end{equation}
Now we apply Lemma $\ref{lem:har}$ to $(\ref{eq:4.3})$. Thus, for any $t\geq 0$, we have
\begin{equation*}
t|\{v>t\}\cap B_1|^{\frac{1}{\epsilon_3}}\leq\left(\int_{B_{1}}v^{\epsilon_3} dx\right)^{\frac{1}{\epsilon_3}}\leq C(\inf_{B_{1}}v+C_1r^{\sigma})\leq C(v(0)+C_1r^{\sigma}).
\end{equation*}
Then
\begin{equation*}
r^{-n}|\{u>t\}\cap B_r|\leq |\{v>t\}\cap B_1|\leq C(v(0)+C_1r^{\sigma})^{\epsilon_3}t^{-\epsilon_3}=C(u(0)+C_1r^{\sigma})^{\epsilon_3}t^{-\epsilon_3}.
\end{equation*}
Therefore, $(\ref{eq:har})$ holds.
\end{proof}
Then we follow the idea in \cite{LL1} to obtain a H\"older estimate.
\begin{theorem}\label{thm:hol}
Assume $0<\sigma_0\leq\sigma<2$, $C_0\geq 0$, and further assume $C_0=0$ if $\sigma<1$. For any $\epsilon>0$, let $\mathcal{F}$ be a class of bounded continuous functions $u$ in $\mathbb R^n$ such that, $-\frac{1}{2}\leq u\leq\frac{1}{2}$ in $\mathbb R^n$, $u$ is a viscosity subsolution of $M_{\mathcal{L}}^+u+C_0|\nabla u|=-\frac{\epsilon}{2}$ in $B_1$, $w=\sup_{u\in\mathcal{F}}u$ is a discontinuous viscosity supersolution of $M_{\mathcal{L}}^-w-C_0|\nabla w|=\frac{\epsilon}{2}$ in $B_1$. Then there exist constants $\epsilon_4$, $\alpha$ and $C$ such that, if $\epsilon<\epsilon_4$,
\begin{equation*}
-C|x|^\alpha\leq w_*(x)-w^*(0)\leq w^*(x)-w_*(0)\leq C|x|^\alpha,
\end{equation*}
where $\epsilon_4$, $\alpha$ and $C$ depend on $\sigma_0$, $\lambda$, $\Lambda$, $C_0$, $n$ and $\mu$.
\end{theorem}
\begin{proof}
We claim that there exist an increasing sequence $\{m_k\}_k$ and a decreasing sequence $\{M_k\}_k$ such that $M_k-m_k=8^{-\alpha k}$ and $m_k\leq\inf_{B_{8^{-k}}}w_*\leq \sup_{B_{8^{-k}}}w^*\leq M_k$. We will prove this claim by induction.

For $k=0$, we choose $m_0=-\frac{1}{2}$ and $M_0=\frac{1}{2}$ since $-\frac{1}{2}\leq u\leq \frac{1}{2}$ for any $u\in\mathcal{F}$. Assume that we have the sequences up to $m_k$ and $M_k$. In $B_{8^{-k-1}}$, we have either
\begin{equation}\label{eq4.1}
|\{w_*\geq \frac{M_k+m_k}{2}\}\cap B_{8^{-k-1}}|\geq \frac{|B_{8^{-k-1}}|}{2},
\end{equation}
or
\begin{equation}\label{eq4.2}
|\{w_*\leq \frac{M_k+m_k}{2}\}\cap B_{8^{-k-1}}|\geq \frac{|B_{8^{-k-1}}|}{2}.
\end{equation}

Case 1: $(\ref{eq4.1})$ holds. 

We define
\begin{equation*}
v(x):=\frac{w_*(8^{-k}x)-m_k}{\frac{M_k-m_k}{2}}.
\end{equation*}
Thus, $v\geq 0$ in $B_1$ and 
\begin{equation*}
|\{v\geq 1\}\cap B_{\frac{1}{8}}|\geq \frac{|B_{\frac{1}{8}}|}{2}.
\end{equation*}
Since $w$ is a discontinuous viscosity supersolution of $M_{\mathcal{L}}^-w-C_0|\nabla w|=\frac{\epsilon}{2}$ in $B_1$, then $v$ is a viscosity supersolution of 
\begin{equation*}
M_{\mathcal{L}}^-v-C_08^{k(1-\sigma)}|\nabla v|=8^{k(\alpha-\sigma)}\epsilon\quad \text{in $B_{8^k}$}.\end{equation*}
We notice that $C_0=0$ if $\sigma<1$ and choose $\alpha<\sigma_0$. Thus, for any $0<\sigma<2$, $v$ is a viscosity supersolution of
\begin{equation*}
M_{\mathcal{L}}^-v-C_0|\nabla v|=\epsilon\quad \text{in $B_{8^k}$}.
\end{equation*}
By the inductive assumption, we have, for any $k\geq j\geq 0$,
\begin{equation}\label{eq4.3}
v\geq\frac{m_{k-j}-m_k}{\frac{M_k-m_k}{2}}\geq\frac{m_{k-j}-M_{k-j}+M_k-m_k}{\frac{M_k-m_k}{2}}= 2(1-8^{\alpha j})\quad \text{in $B_{8^{j}}$}.
\end{equation}
Moreover, we have 
\begin{equation}\label{eq4.4}
v\geq 2\cdot 8^{\alpha k}[-\frac{1}{2}-(\frac{1}{2}-8^{-\alpha k})]=2(1-8^{\alpha k})\quad \text{in $B_{8^{k}}^c$}.
\end{equation}
By $(\ref{eq4.3})$ and $(\ref{eq4.4})$, we have 
\begin{equation*}
v(x)\geq -2(|8x|^\alpha-1),\quad \text{for any $x\in B_1^c$}.
\end{equation*}
We define 
\begin{equation*}
v^+(x):=\max\{v(x),0\}\quad\text{and}\quad v^-(x):=-\min\{v(x),0\}.
\end{equation*}
Since $v\geq 0$ in $B_1$, $v^-(x)=0$ and $\nabla v^-(x)=0$ for any $x\in B_1$. By {\rm(H1)}, we can choose sufficiently small $\alpha$ independent of $\sigma$ such that, for any $x\in B_{\frac{3}{4}}$ and $\sigma_0\leq\sigma<2$,
\begin{eqnarray*}
M_{\mathcal{L}}^-v^+(x)&\leq& M_{\mathcal{L}}^-v(x)+M_{\mathcal{L}}^+v^-(x)\\
&\leq&M_{\mathcal{L}}^-v(x)+\sup_{K\in\mathcal{L}}\int_{\mathbb R^n}\delta_zv^-(x)K(z)dz\\
&\leq&M_{\mathcal{L}}^-v(x)+\sup_{K\in\mathcal{L}}\int_{B_{\frac{1}{4}}^c\cap\{v(x+z)<0\}} v^-(x+z)K(z)dz\\
&\leq&M_{\mathcal{L}}^-v(x)+\sup_{K\in\mathcal{L}}\int_{B_{\frac{1}{4}}^c}\max\{2(|8(x+z)|^{\alpha}-1),0\}K(z)dz\\
&\leq&M_{\mathcal{L}}^-v(x)+2(2-\sigma)\Lambda\sum_{l=0}^{+\infty}\left(\frac{2^l}{4}\right)^{-\sigma}\left(2^{(l+4)\alpha}-1\right)\\
&\leq&M_{\mathcal{L}}^-v(x)+2^{13}(2-\sigma_0)\Lambda\left(\frac{2^{4(\alpha-\sigma_0)}}{1-2^{\alpha-\sigma_0}}-\frac{2^{-4\sigma_0}}{1-2^{-\sigma_0}}\right)\\
&\leq&M_{\mathcal{L}}^-v(x)+\epsilon.
\end{eqnarray*}
Therefore, we have 
\begin{equation*}
M_{\mathcal{L}}^-v^+-C_0|\nabla v^+|\leq 2\epsilon,\quad\text{in $B_{\frac{3}{4}}$}.
\end{equation*}
Given any point $x\in B_{\frac{1}{8}}$, we can apply Corollary $\ref{cor:har}$ in $B_{\frac{1}{4}}(x)$ to obtain
\begin{equation*}
C(v^+(x)+2\epsilon)^{\epsilon_3}\geq|\{v^+>1\}\cap B_{\frac{1}{4}}(x)|\geq |\{v^+>1\}\cap B_{\frac{1}{8}}|\geq\frac{|B_{\frac{1}{8}}|}{2}.
\end{equation*}
Thus, we can choose sufficiently small $\epsilon_4$ such that $v^+\geq \epsilon_4$ in $B_{\frac{1}{8}}$ if $\epsilon<\epsilon_4$. Therefore,
\begin{equation*}
v(x)=\frac{w_*(8^{-k}x)-m_k}{\frac{M_k-m_k}{2}}\geq \epsilon_4\quad \text{in $B_{\frac{1}{8}}$}.
\end{equation*}
If we set $m_{k+1}=m_k+\epsilon_4\frac{M_k-m_k}{2}$ and $M_{k+1}=M_k$, we must have $m_{k+1}\leq\inf_{B_{8^{-k-1}}}w_*\leq \sup_{B_{8^{-k-1}}}w^*\leq M_{k+1}$. 

Case 2: $(\ref{eq4.2})$ holds. 

For any $u\in\mathcal{F}$, we obtain that $u\in C^0(\mathbb R^n)$ is a viscosity subsolution of $M_{\mathcal{L}}^+u+C_0|\nabla u|=-\frac{\epsilon}{2}$ in $B_1$ and $u\leq w_*$ in $\mathbb R^n$. Thus, we have 
\begin{equation*}
|\{u\leq\frac{M_k+m_k}{2}\}\cap B_{8^{-k-1}}|\geq \frac{|B_{8^{-k-1}}|}{2}.
\end{equation*}
We define
\begin{equation*}
v_{ u}(x):=\frac{M_k-u(8^{-k}x)}{\frac{M_k-m_k}{2}}.
\end{equation*}
Thus, $v_{ u}\geq 0$ in $B_1$ and 
\begin{equation*}
|\{v_{ u}\geq 1\}\cap B_{\frac{1}{8}}|\geq \frac{|B_{\frac{1}{8}}|}{2}.
\end{equation*}
Since $ u$ is a viscosity subsolution of $M_{\mathcal{L}}^+ u+C_0|\nabla u|=-\frac{\epsilon}{2}$ in $B_1$,
then $v_{ u}$ is a viscosity supersolution of 
\begin{equation*}
M_{\mathcal{L}}^-v_{ u}-C_0|\nabla v_u|= \epsilon\quad \text{in $B_{8^k}$}.
\end{equation*}
Similar to Case 1, we have, if $\epsilon<\epsilon_4$,
\begin{equation*}
v_{ u}(x)=\frac{M_k- u(8^{-k}x)}{\frac{M_k-m_k}{2}}\geq \epsilon_4\quad \text{in $B_{\frac{1}{8}}$},
\end{equation*}
which implies
\begin{equation*}
 u(8^{-k}x)\leq M_k-\epsilon_4\frac{M_k-m_k}{2}\quad \text{in $B_{\frac{1}{8}}$}.
\end{equation*}
By the definition of $w$, we have 
\begin{equation*}
w^*(8^{-k}x)\leq M_k-\epsilon_4\frac{M_k-m_k}{2}\quad \text{in $B_{\frac{1}{8}}$}.
\end{equation*}
If we set $m_{k+1}=m_k$ and $M_{k+1}=M_k-\epsilon_4\frac{M_k-m_k}{2}$, we must have $m_{k+1}\leq\inf_{B_{8^{-k-1}}}w_*\leq \sup_{B_{8^{-k-1}}}w^*\leq M_{k+1}$. 

Therefore, in both of the cases, we have $M_{k+1}-m_{k+1}=(1-\frac{\epsilon_4}{2})8^{-\alpha k}$. We then choose $\alpha$ and $\epsilon_4$ sufficiently small such that $(1-\frac{\epsilon_4}{2})=8^{-\alpha}$. Thus we have $M_{k+1}-m_{k+1}=8^{-\alpha(k+1)}$.
\end{proof}

\begin{theorem}\label{cor:hol1}
Assume that $0<\sigma_0\leq\sigma<2$ and $I(x,0,0)$ is bounded in $\Omega$. Assume that $I$ is uniformly elliptic and satisfies {\rm (A0)}, {\rm (A2)}. Let $w$ be the bounded discontinuous viscosity solution
of $(\ref{eq:gennon})$ constructed in Theorem $\ref{thm:per}$. Then, for any sufficiently small $\tilde \delta>0$, there exists a constant $C$ such that $w\in C^{\alpha}(\Omega)$ and 
\begin{equation*}
\|w\|_{C^\alpha(\bar\Omega_{\tilde\delta})}\leq C(C_2+m(C_2)+\|I(\cdot,0,0)\|_{L^{\infty}(\Omega)}),
\end{equation*}
where $\alpha$ is given in Theorem $\ref{thm:hol}$, $C_2:=\max\{\|\underbar u\|_{L^{\infty}(\mathbb R^n)},\|\bar u\|_{L^{\infty}(\mathbb R^n)}\}$ and $C$ depends on $\sigma_0$, $\tilde\delta$, $\lambda$, $\Lambda$, $C_0$, $n$, $\mu$.
\end{theorem}
\begin{proof}
It is obvious that $\|u\|_{L^{\infty}(\mathbb R^n)}\leq C_2$ if $u\in\mathcal{F}$. Since $I$ is uniformly elliptic, we have
\begin{eqnarray*}
I(x,0,0)-I(x,u(x),u(\cdot))\leq M_{\mathcal{L}}^+u(x)+C_0|\nabla u(x)|+m(C_2),\quad \text{in $\Omega$}.
\end{eqnarray*} 
Since $u$ is a viscosity subsolution of $I=0$ in $\Omega$, we have
\begin{equation*}
-m(C_2)-\|I(\cdot,0,0)\|_{L^{\infty}(\Omega)}\leq M_{\mathcal{L}}^+u+C_0|\nabla u|,\quad \text{in $\Omega$}.
\end{equation*} 
Similarly, we have 
\begin{equation*}
M_{\mathcal{L}}^-w_*-C_0|\nabla w_*|\leq m(C_2)+\|I(\cdot,0,0)\|_{L^{\infty}(\Omega)},\quad \text{in $\Omega$}.
\end{equation*}
By normalization, the result follows from Theorem $\ref{thm:hol}$.
\end{proof}

By applying Theorem $\ref{cor:hol1}$ to Bellman-Isaacs equation, we have the following corollary.
\begin{corollary}\label{cor:hol2}
Assume that $0<\sigma_0\leq\sigma<2$, $b_{ab}\equiv 0$ in $\Omega$ if $\sigma<1$ and $c_{ab}\geq 0$ in $\Omega$. Assume that $\{K_{ab}(\cdot,z)\}_{a,b,z}$, $\{b_{ab}\}_{a,b}$, $\{c_{ab}\}_{a,b}$, $\{f_{ab}\}_{a,b}$ are sets of uniformly  bounded and continuous functions in $\Omega$, uniformly in $a\in\mathcal{A}$, $b\in\mathcal{B}$, and $\{K_{ab}(x,\cdot):x\in\Omega, a\in\mathcal{A}, b\in\mathcal{B}\}$ are kernels satisfying {\rm (H0)}-{\rm (H3)}. Let $w$ be the bounded discontinuous viscosity solution of $(\ref{eq:belisa})$ constructed in Corollary $\ref{col:biper}$. Then, for any sufficiently small $\tilde \delta>0$, there exists a constant $C$ such that $w\in C^{\alpha}(\Omega)$ and
\begin{equation*}
\|w\|_{C^\alpha(\bar \Omega_{\tilde\delta})}\leq C(C_2+\sup_{a\in\mathcal{A},b\in\mathcal{B}}\|f_{ab}\|_{L^{\infty}(\Omega)}),
\end{equation*}
where $\alpha$, $C_2$ are given in Theorem $\ref{cor:hol1}$ and $C$ depends on $\sigma_0$, $\tilde\delta$, $\lambda$, $\Lambda$, $\sup_{a\in\mathcal{A},b\in\mathcal{B}}\|b_{ab}\|_{L^{\infty}(\Omega)}$, $\sup_{a\in\mathcal{A},b\in\mathcal{B}}\|c_{ab}\|_{L^{\infty}(\Omega)}$, $n$, $\mu$.
\end{corollary}
\begin{remark}
In this section we assume our nonlocal equations satisfy the weak uniform ellipticity introduced in \cite{SS} mainly because, to our knowledge, this is the weakest assumption to get the weak Harnack inequality. In fact, our approach to get H\"older continuity of the discontinuous viscosity solution constructed by Perron's method could be applied to more general nonlocal equations as long as the weak Harnack inequality holds for such equation.
\end{remark}

\section{Continuous sub/supersolutions}

In this section we construct continuous sub/supersolutions in both uniformly elliptic and degenerate cases.

\subsection{Uniformly elliptic case}
In the uniformly elliptic case, we follow the idea in \cite{XJ1} to establish barrier functions. We define $v_\alpha(x)=((x_1-1)^+)^{\alpha}$ where $0<\alpha<1$ and $x=(x_1,x_2,...,x_n)$.

\begin{lemma}\label{lem:bar1}
Assume that $0<\sigma<2$. Then there exists a sufficiently small $\alpha>0$ such that  $M_{\mathcal{L}}^+v_{\alpha}((1+r)e_1)\leq-\epsilon_5r^{\alpha-\sigma}$ for any $r>0$ where $e_1=(1,0,...,0)$ and $\epsilon_5$ is some positive constant.
\end{lemma}
\begin{proof}
Case 1: $0<\sigma<1$. 

By Lemma 2.2 in \cite{SS}, we have, for any $r>0$ and $\alpha>0$,
\begin{eqnarray*}
M_{\mathcal{L}}^+v_{\alpha}((1+r)e_1)&=&\sup_{K\in\mathcal{L}}\int_{\mathbb R^n}\left(v_{\alpha}\left(\left(1+r\right)e_1+z\right)-v_{\alpha}((1+r)e_1)\right)K(z)dz\\
&=&\sup_{K\in\mathcal{L}}\int_{\mathbb R^n}\left(\left(\left(r+z_1\right)^+\right)^{\alpha}-r^\alpha\right)K(z)dz\\
&=&r^{\alpha-\sigma}\sup_{K\in\mathcal{L}}\int_{\mathbb R^n}\left(\left(\left(1+z_1\right)^+\right)^{\alpha}-1\right)r^{n+\sigma}K(rz)dz\\
&=&r^{\alpha-\sigma}\sup_{K\in\mathcal{L}}\int_{\mathbb R^n}\left(\left(\left(1+z_1\right)^+\right)^{\alpha}-1\right)K(z)dz\\
&\leq&r^{\alpha-\sigma}\left(\sup_{K\in\mathcal{L}}\int_{z_1>-1}\left(\left(1+z_1\right)^{\alpha}-1\right)K(z)dz-\inf_{K\in\mathcal{L}}\int_{z_1\leq-1}K(z)dz\right).
\end{eqnarray*}
By {\rm (H3)}, we have, for any $K\in\mathcal{L}$ and any $\delta>0$, there is a set $A_\delta$ satisfying $A_\delta\subset B_{2\delta}\setminus B_\delta$, $A_\delta=-A_\delta$, $|A_\delta|\geq \mu|B_{2\delta}\setminus B_\delta|$ and $K(z)\geq(2-\sigma)\lambda \delta^{-n-\sigma}$ in $A_\delta$. It is obvious that
\begin{equation*}
\mu_\delta:=\frac{|(B_{2\delta}\setminus B_\delta)\cap\{z;|z_1|< 1\}|}{|B_{2\delta}\setminus B_\delta|}\to 0\quad\text{as $\delta\to+\infty$}.
\end{equation*}
Thus, there exists $\delta_3>0$ such that $\mu_{\delta}<\frac{\mu}{2}$ if $\delta\geq\delta_3$. Then 
\begin{equation*}
\frac{|\{z;|z_1|\geq 1\}\cap A_{\delta_3}|}{|B_{2\delta_3}\setminus B_{\delta_3}|}\geq \frac{|A_{\delta_3}|-|(B_{2\delta_3}\setminus B_{\delta_3})\cap\{z;|z_1|< 1\}|}{|B_{2\delta_3}\setminus B_{\delta_3}|}\geq\frac{\mu}{2}.
\end{equation*}
By the symmetry of $A_{\delta_3}$, we have
\begin{equation*}
\frac{|\{z;z_1\leq -1\}\cap A_{\delta_3}|}{|B_{2\delta_3}\setminus B_{\delta_3}|}\geq \frac{\mu}{4}.
\end{equation*}
Therefore, we have, for any $K\in\mathcal{L}$,
\begin{equation}\label{eq:5.1}
\int_{z_1\leq -1}K(z)dz\geq \int_{\{z;z_1\leq -1\}\cap A_{\delta_3}}K(z)dz\geq  \frac{(2-\sigma)\lambda\mu}{4}\delta_3^{-n-\sigma}|B_{2\delta_3}\setminus B_{\delta_3}|=:2\epsilon_5.
\end{equation}
By {\rm (H1)} and {\rm (H2)}, we have, for any $K\in\mathcal{L}$,
\begin{eqnarray}
\int_{z_1>-1}((1+z_1)^\alpha-1)K(z)dz&=&\int_{\{z;z_1>-1\}\cap B_{\frac{1}{2}}}+\int_{\{z;z_1>-1\}\cap B_{\frac{1}{2}}^c}\nonumber\\
&\leq&\alpha 2^{1-\alpha}|\int_{B_{\frac{1}{2}}}zK(z)dz|+\int_{\{z;z_1>-1\}\cap B_{\frac{1}{2}}^c}((1+z_1)^\alpha-1)K(z)dz\nonumber\\
&\leq&\alpha2^{1-\alpha}(1-\sigma)\Lambda\sum_{l=0}^{+\infty}\left(\frac{1}{2^{l+2}}\right)^{1-\sigma}\nonumber\\
&&+(2-\sigma)\Lambda\sum_{l=0}^{+\infty}(2^{l-1})^{-\sigma}\left((1+2^l)^\alpha-1\right)\nonumber\\
&\leq&2\alpha\Lambda\frac{1-\sigma}{1-2^{\sigma-1}}+8\Lambda \left(\frac{2^{\alpha-\sigma}}{1-2^{\alpha-\sigma}}-\frac{2^{-\sigma}}{1-2^{-\sigma}}\right).\label{eq::5.2}
\end{eqnarray}
Thus, we have
\begin{equation*}
\lim_{\alpha\to 0^+}\sup_{K\in\mathcal{L}}\int_{z_1>-1}\left(\left(1+z_1\right)^{\alpha}-1\right)K(z)dz-\inf_{K\in\mathcal{L}}\int_{z_1\leq-1}K(z)dz\leq -2\epsilon_5.
\end{equation*}
Then there exists a sufficiently small $\alpha$ such that
\begin{equation*}
M_{\mathcal{L}}^+v_\alpha((1+r)e_1)\leq-\epsilon_5 r^{\alpha-\sigma}.
\end{equation*}

Case 2: $\sigma=1$. 

Using {\rm (H2)}, we have, for any $r>0$ and $\alpha>0$,
\begin{eqnarray*}
&&M_{\mathcal{L}}^+v_{\alpha}((1+r)e_1)\\
&=&\sup_{K\in\mathcal{L}}\int_{\mathbb R^n}\left(v_{\alpha}\left(\left(1+r\right)e_1+z\right)-v_{\alpha}((1+r)e_1)-\mathbbm{1}_{B_1}(z)\nabla v_\alpha((1+r)e_1)\cdot z\right)K(z)dz\\
&=&\sup_{K\in\mathcal{L}}\int_{\mathbb R^n}\left(\left(\left(r+z_1\right)^+\right)^{\alpha}-r^\alpha-\mathbbm{1}_{B_1}(z)\alpha r^{\alpha-1}z_1\right)K(z)dz\\
&=&r^{\alpha-1}\sup_{K\in\mathcal{L}}\int_{\mathbb R^n}\left(\left(\left(1+z_1\right)^+\right)^{\alpha}-1-\mathbbm{1}_{B_{\frac{1}{r}}}(z)\alpha z_1\right)r^{n+1}K(rz)dz\\
&=&r^{\alpha-1}\sup_{K\in\mathcal{L}}\int_{\mathbb R^n}\left(\left(\left(1+z_1\right)^+\right)^{\alpha}-1-\mathbbm{1}_{B_\frac{1}{2}}(z)\alpha z_1\right)K(z)dz\\
&\leq&r^{\alpha-1}\left(\sup_{K\in\mathcal{L}}\int_{z_1>-1}\left(\left(1+z_1\right)^{\alpha}-1-\mathbbm{1}_{B_\frac{1}{2}}(z)\alpha z_1\right)K(z)dz-\inf_{K\in\mathcal{L}}\int_{z_1\leq-1}K(z)dz\right).
\end{eqnarray*}
By {\rm (H1)}, we have, for any $K\in\mathcal{L}$,
\begin{eqnarray*}
&&\int_{z_1>-1}((1+z_1)^\alpha-1-\mathbbm{1}_{B_\frac{1}{2}}(z)\alpha z_1)K(z)dz\\
&=&\int_{\{z;z_1>-1\}\cap B_{\frac{1}{2}}}((1+z_1)^\alpha-1-\alpha z_1)K(z)dz+\int_{\{z;z_1>-1\}\cap B_{\frac{1}{2}}^c}((1+z_1)^\alpha-1)K(z)dz\\
&\leq&\alpha(1-\alpha) 2^{2-\alpha}\int_{B_{\frac{1}{2}}}|z|^2K(z)dz+\int_{\{z;z_1>-1\}\cap B_{\frac{1}{2}}^c}((1+z_1)^\alpha-1)K(z)dz\\
&\leq&\alpha(1-\alpha)2^{2-\alpha}\Lambda\sum_{l=0}^{+\infty}\left(\frac{1}{2^{l+2}}\right)^{-1}\left(\frac{1}{2^{l+1}}\right)^2+\Lambda\sum_{l=0}^{+\infty}(2^{l-1})^{-1}\left((1+2^l)^\alpha-1\right)\\
&\leq&8\alpha \Lambda+4\Lambda \left(\frac{2^{\alpha-1}}{1-2^{\alpha-1}}-\frac{2^{-1}}{1-2^{-1}}\right).
\end{eqnarray*}
Then the rest of proof is similar to Case 1.

Case 3: $1<\sigma<2$. 

For any $r>0$ and $\alpha>0$, we have
\begin{eqnarray*}
&&M_{\mathcal{L}}^+v_{\alpha}((1+r)e_1)\\
&=&\sup_{K\in\mathcal{L}}\int_{\mathbb R^n}\left(v_{\alpha}\left(\left(1+r\right)e_1+z\right)-v_{\alpha}((1+r)e_1)-\nabla v_\alpha((1+r)e_1)\cdot z\right)K(z)dz\\
&=&\sup_{K\in\mathcal{L}}\int_{\mathbb R^n}\left(\left(\left(r+z_1\right)^+\right)^{\alpha}-r^\alpha-\alpha r^{\alpha-1}z_1\right)K(z)dz\\
&=&r^{\alpha-\sigma}\sup_{K\in\mathcal{L}}\int_{\mathbb R^n}\left(\left(\left(1+z_1\right)^+\right)^{\alpha}-1-\alpha z_1\right)K(z)dz\\
&\leq&r^{\alpha-\sigma}\left(\sup_{K\in\mathcal{L}}\int_{z_1>-1}\left(\left(\left(1+z_1\right)^+\right)^{\alpha}-1-\alpha z_1\right)K(z)dz-\inf_{K\in\mathcal{L}}\int_{z_1\leq-1}\left(1+\alpha z_1\right)K(z)dz\right).
\end{eqnarray*}
Using $(\ref{eq:5.1})$ and {\rm (H2)}, we have
\begin{eqnarray*}
\inf_{K\in\mathcal{L}}\int_{z_1\leq-1}\left(1+\alpha z_1\right)K(z)dz
\geq\inf_{K\in\mathcal{L}}\int_{z_1\leq-1}K(z)dz-\alpha\sup_{K\in\mathcal{L}}|\int_{B_1^c} zK(z)dz|\geq2\epsilon_5-\frac{\alpha\Lambda(\sigma-1)}{1-2^{1-\sigma}}.
\end{eqnarray*}
By {\rm (H1)} and {\rm (H2)}, we have, for any $K\in\mathcal{L}$,
\begin{eqnarray*}
&&\int_{z_1>-1}((1+z_1)^\alpha-1-\alpha z_1)K(z)dz
=\int_{\{z;z_1>-1\}\cap B_{\frac{1}{2}}}+\int_{\{z;z_1>-1\}\cap B_{\frac{1}{2}}^c}\\
&&\qquad\qquad\qquad\qquad\qquad\leq\alpha(1-\alpha) 2^{2-\alpha}\int_{B_{\frac{1}{2}}}|z|^2K(z)dz+\alpha|\int_{\{z;z_1>-1\}\cap B_{\frac{1}{2}}^c}zK(z)dz|\\
&&\quad\qquad\qquad\qquad\qquad\qquad+\int_{\{z;z_1>-1\}\cap B_{\frac{1}{2}}^c}((1+z_1)^\alpha-1)K(z)dz\\
&&\qquad\qquad\qquad\qquad\qquad\leq\frac{16\alpha(2-\sigma)\Lambda}{1-2^{\sigma-2}}+\frac{2\alpha\Lambda(\sigma-1)}{1-2^{1-\sigma}}+16(2-\sigma)\Lambda\left(\frac{2^{\alpha-\sigma}}{1-2^{\alpha-\sigma}}-\frac{2^{-\sigma}}{1-2^{-\sigma}}\right).
\end{eqnarray*}
Then we have
\begin{eqnarray*}
&&\lim_{\alpha\to 0^+}\sup_{K\in\mathcal{L}}\int_{z_1>-1}\left(\left(\left(1+z_1\right)^+\right)^{\alpha}-1-\alpha z_1\right)K(z)dz-\inf_{K\in\mathcal{L}}\int_{z_1\leq-1}\left(1+\alpha z_1\right)K(z)dz\\
&\leq&\lim_{\alpha\to 0^+}\frac{16\alpha(2-\sigma)\Lambda}{1-2^{\sigma-2}}+\frac{2\alpha\Lambda(\sigma-1)}{1-2^{1-\sigma}}+16(2-\sigma)\Lambda\left(\frac{2^{\alpha-\sigma}}{1-2^{\alpha-\sigma}}-\frac{2^{-\sigma}}{1-2^{-\sigma}}\right)-2\epsilon_5+\frac{\alpha\Lambda(\sigma-1)}{1-2^{1-\sigma}}\\
&=&-2\epsilon_5.
\end{eqnarray*}
Similar to Case 1, there exists a sufficiently small $\alpha$ such that
\begin{equation*}
M_{\mathcal{L}}^+ v_\alpha((1+r)e_1)\leq-\epsilon_5r^{\alpha-\sigma}.
\end{equation*}
\end{proof}

\begin{lemma}\label{lem:bar2}
Assume that $0<\sigma<2$, $C_0\geq 0$ and further assume $C_0=0$ if $\sigma<1$. Then there are $\alpha>0$ and $0<r_0<1$ sufficiently small so that the function $u_\alpha(x):=((|x|-1)^+)^\alpha$
satisfies $M_{\mathcal{L}}^+u_\alpha+C_0|\nabla u_{\alpha}|\leq -1$ in $\bar B_{1+r_0}\setminus \bar B_1$.
\end{lemma}
\begin{proof}
We notice that $u_\alpha$ and $|\nabla|$ are rotation invariant. By Lemma 2.2 in \cite{SS}, $M_{\mathcal{L}}^+$ is also rotation invariant. Then we only need to prove that
$M_{\mathcal{L}}^+u_\alpha((1+r)e_1)+C_0|\nabla u_{\alpha}((1+r)e_1)|\leq -1$ for any $r\in (0,r_0]$ where $r_0$ and $\alpha$ are sufficiently small positive constants.
Note that, $\forall r>0$, $u_\alpha((1+r)e_1)=v_\alpha((1+r)e_1)$, $\nabla u_\alpha((1+r)e_1)=\nabla v_\alpha((1+r)e_1)$ and that 
\begin{equation*}
|(|(1+r)e_1+z|-1)^+-(r+z_1)^+|\leq C|z'|^2,\quad\text{for any $z\in B_1$},
\end{equation*}
where $z=(z_1,z')$. Therefore, we have
\begin{equation*}
0\leq (u_\alpha-v_\alpha)((1+r)e_1+z)\leq\left\{\begin{array}{ll} Cr^{\alpha-1}|z'|^2,
\quad z\in B_\frac{r}{2},\\
C|z'|^{2\alpha},\qquad\,\,\, z\in B_1\setminus B_{\frac{r}{2}},\\
C|z|^{\alpha}, \qquad\quad\, z\in \mathbb R^n\setminus B_1.
\end{array}
  \right.
\end{equation*}
Using {\rm (H1)}, we have, for any $0<\sigma<2$ and $L\in \mathcal{L}$,
\begin{eqnarray*}
0&\leq&L(u_\alpha-v_\alpha)((1+r)e_1)\\
&=&\int_{\mathbb R^n}(u_\alpha-v_\alpha)((1+r)e_1+z)K(z)dz\\
&\leq& C\left(\int_{B_{\frac{r}{2}}}r^{\alpha-1}|z'|^2 K(z)dz
+\int_{B_1\setminus B_{\frac{r}{2}}}|z'|^{2\alpha}K(z)dz
+\int_{\mathbb R^n\setminus B_1}|z|^{\alpha}K(z)dz\right)\\
&\leq& C\left(\int_{B_{\frac{r}{2}}}r^{\alpha-1}|z|^2 K(z)dz
+\int_{B_{\frac{r}{2}}^c}|z|^{2\alpha}K(z)dz\right)\\
&\leq& C\Lambda( r^{\alpha-\sigma+1}+r^{2\alpha-\sigma}).
\end{eqnarray*}
Thus, we have $M_{\mathcal{L}}^+(u_\alpha-v_\alpha)((1+r)e_1)\leq C\Lambda(r^{\alpha-\sigma+1}+r^{2\alpha-\sigma})$.
Therefore, by Lemma $\ref{lem:bar1}$, there exists a sufficiently small $\alpha>0$ such that
\begin{eqnarray*}
&&M_{\mathcal{L}}^+u_\alpha((1+r)e_1)+C_0|\nabla u_\alpha((1+r)e_1)|\\
&\leq& M_{\mathcal{L}}^+(u_\alpha-v_\alpha)((1+r)e_1)+M_{\mathcal{L}}^+v_\alpha((1+r)e_1)+C_0|\nabla u_\alpha((1+r)e_1)|\\
&\leq& C\Lambda(r^{\alpha-\sigma+1}+r^{2\alpha-\sigma})-\epsilon_5 r^{\alpha-\sigma}+\alpha C_0 r^{\alpha-1}.
\end{eqnarray*}
We notice that $\alpha-\sigma+1>\alpha-\sigma$, $2\alpha-\sigma>\alpha-\sigma$ and
\begin{itemize}
\item[(i)] if $0<\sigma<1$, then $C_0=0$;
\item[(ii)] if $\sigma=1$, then $\alpha C_0\to 0$ as $\alpha\to 0$;
\item[(iii)] if $1<\sigma<2$, then $\alpha-1>\alpha-\sigma$.
\end{itemize}
Thus, there exist sufficiently small $0<r_0<1$ such that we have, for any $r\in (0,r_0]$,
\begin{equation}\label{eq:::5.3}
M_{\mathcal{L}}^+ u_\alpha((1+r)e_1)+C_0|\nabla u_\alpha((1+r)e_1)|\leq -1.
\end{equation}
\end{proof}

In the rest of this section, we assume that $\Omega$ satisfies the uniform exterior ball condition, i.e., there is a constant $r_\Omega>0$ such that, for any $x\in\partial \Omega$ and $0<r\leq r_\Omega$, there exists $y_x^r\in \Omega^c$ satisfying $\bar B_r(y_x^r)\cap\bar\Omega=\{x\}$. Without loss of generality, we can assume that $r_\Omega<1$. Since $\Omega$ is a bounded domain, there exists a sufficiently large constant $R_0>0$ such that $\Omega\subset \{y;|y_1|<R_0\}$.

\begin{remark}
At this stage, we are not sure about whether the exterior ball condition is necessary for the construction of sub/supersolution. In future work, we plan to construct sub/supersolutions under a weaker assumption on $\Omega$, such as the cone condition.
\end{remark}

\begin{lemma}\label{lem:bar4}
Assume that $0<\sigma<2$, $C_0\geq 0$ and further assume $C_0=0$ if $\sigma<1$. There exists an $\epsilon_7>0$ such that, for any $x\in \partial\Omega$ and $0<r<r_\Omega$, there is a continuous function $\varphi_{x,r}$ satisfying
\begin{equation*}
\left\{\begin{array}{ll} \varphi_{x,r}\equiv 0,
\quad in\,\,\bar B_r(y_x^r),\\
\varphi_{x,r}>0,\quad in\,\, \bar B_r^c(y_x^r),\\
\varphi_{x,r}\geq 1, \quad in\,\, B_{2r}^c(y_x^r),\\
M_{\mathcal{L}}^+\varphi_{x,r}+C_0|\nabla \varphi_{x,r}|\leq -\epsilon_7,\quad in\,\,\Omega.
\end{array}
  \right.
\end{equation*}
\end{lemma}
\begin{proof}
We define a uniformly continuous function $\varphi$ in $\mathbb R^n$ such that $1\leq\varphi\leq 2$ and  
\begin{equation*}
\left\{\begin{array}{ll} \varphi(y)= 1,
\quad in\,\,y_1>R_0+1,\\
\varphi(y)=2,\quad in\,\,y_1\leq R_0.
\end{array}
  \right.
\end{equation*}
We pick some sufficiently large $C_3>\frac{2}{r_0^\alpha}$ and we define $\varphi_{x,r}(y)=\min\{\varphi(y), C_3u_\alpha(\frac{y-y_x^r}{r})\}$ where $\alpha$ and $r_0$ are defined in Lemma $\ref{lem:bar2}$. It is easy to verify that $\varphi_{x,r}\equiv 0$ in $\bar B_r(y_x^r)$, $\varphi_{x,r}>0$ in $\bar B_r^c(y_x^r)$, and $\varphi_{x,r}\geq 1$ in $B_{2r}^c(y_x^r)$. By Lemma $\ref{lem:bar2}$, we have $M_{\mathcal{L}}^+u_\alpha+C_0|\nabla u_\alpha|\leq -1$ in $\bar B_{1+r_0}\setminus \bar B_{1}$. It is obvious that, for any $y\in \bar B_{(1+r_0)r}(y_x^r)\setminus \bar B_{r}(y_x^r)$, we have 
\begin{equation*}
(M_{\mathcal{L}}^+u_\alpha(\frac{\cdot-y_x^r}{r}))(y)+C_0r^{1-\sigma}|(\nabla u_\alpha(\frac{\cdot-y_x^r}{r}))(y)|\leq-r^{-\sigma},\quad\text{for any $0<r<r_\Omega$.}
\end{equation*}
Since $C_0=0$ if $0<\sigma<1$, and $0<r<1$, then
\begin{equation*}
(M_{\mathcal{L}}^+u_\alpha(\frac{\cdot-y_x^r}{r}))(y)+C_0|(\nabla u_\alpha(\frac{\cdot-y_x^r}{r}))(y)|\leq-1,\quad\text{for any $0<r<r_\Omega$.}
\end{equation*}

For any $y\in\bar B_{(1+(\frac{2}{C_3})^{\frac{1}{\alpha}})r}(y_x^r)\setminus \bar B_r(y_x^r)$, we have $\varphi_{x,r}(y)=C_3u_\alpha(\frac{y-y_x^r}{r})$. Suppose that there exists a test function $\psi\in C_b^2(\mathbb R^n)$ touches $\varphi_{x,r}$ from below at $y$. Thus, $\frac{\psi}{C_3}$ touches $u_\alpha(\frac{\cdot-y_x^r}{r})$ from below at $y$. Thus, $M_{\mathcal{L}}^+\psi(y)+C_0|\nabla \psi(y)|\leq-C_3$. For any $y\in\Omega\cap \bar B_{(1+(\frac{2}{C_3})^{\frac{1}{\alpha}})r}^c(y_x^r)$, we have $\varphi_{x,r}(y)=\varphi(y)=\max_{\mathbb R^n}\varphi_{x,r}=2$. Therefore, for any $0<\sigma<2$, we have

\begin{eqnarray*}
(M_{\mathcal{L}}^+\varphi_{x,r})(y)+C_0|\nabla \varphi_{x,r}(y)|&=&\sup_{K\in\mathcal{L}}\int_{\mathbb R^n}\left(\varphi_{x,r}(y+z)-\varphi_{x,r}(y)\right)K(z)dz\nonumber\\
&=&\sup_{K\in\mathcal{L}}\int_{\mathbb R^n}(\varphi_{x,r}(y+z)-2)K(z)dz\nonumber\\
&\leq&-\inf_{K\in\mathcal{L}}\int_{\{z|z_1>-y_1+R_0+1\}}K(z)dz\nonumber\\
&\leq&-\inf_{K\in\mathcal{L}}\int_{\{z|z_1>2R_0+1\}}K(z)dz.
\end{eqnarray*}
By a similar estimate to $(\ref{eq:5.1})$, there exists a positive constant $\epsilon_6$ such that, for any $K\in\mathcal{L}$, we have
\begin{equation*}
\int_{\{z|z_1>2R_0+1\}}K(z)dz\geq \epsilon_6.
\end{equation*}
Then, for any $y\in\Omega\cap \bar B_{(1+(\frac{2}{C_3})^{\frac{1}{\alpha}})r}^c(y_x^r)$, we have
\begin{equation}\label{eq:5.2}
M_{\mathcal{L}}^+\varphi_{x,r}(y)+C_0|\nabla \varphi_{x,r}(y)|\leq -\epsilon_6.
\end{equation}
Based on the above estimates, if we set $\epsilon_7=\min\{C_3,\epsilon_6\}$, we have 
\begin{equation*}
M_{\mathcal{L}}^+\varphi_{x,r}+C_0|\nabla\varphi_{x,r}| \leq -\epsilon_7,\quad in\,\,\Omega.
\end{equation*}
\end{proof}

\begin{theorem}\label{thm:supsub1}
Assume that $0<\sigma<2$, $I(x,0,0)$ is bounded in $\Omega$ and $g$ is a bounded continuous function in $\mathbb R^n$. Assume that $I$ is uniformly elliptic and satisfies {\rm (A0)}, {\rm (A2)}. Then $(\ref{eq:gennon})$ admits a continuous viscosity supersolution $\bar u$ and a continuous viscosity subsolution $\underbar u$ and $\bar u=\underbar u=g$ in $\Omega^c$. 
\end{theorem}
\begin{proof}
We only prove $(\ref{eq:gennon})$ admits a viscosity supersolution $\bar u$ and $\bar u=g$ in $\Omega^c$. For a viscosity subsolution, the construction is similar. Since $I$ is uniformly elliptic, we have, for any $x\in\Omega$,
 \begin{equation*}
-m(\|g\|_{L^{\infty}(\mathbb R^n)})\leq I(x,-\|g\|_{L^{\infty}(\mathbb R^n)},0)-I(x,0,0)\leq m(\|g\|_{L^{\infty}(\mathbb R^n)}).
 \end{equation*}
Thus, we have $\|I(\cdot,-\|g\|_{L^{\infty}(\mathbb R^n)},0)\|_{L^{\infty}(\Omega)}<+\infty$. 
Since $g$ is a continuous function, let $\rho_R$ be a modulus of continuity of $g$ in $B_R$. Let $R_1$ be a sufficiently large constant such that $\Omega\subset B_{R_1-1}$. For any $x\in\partial\Omega$, we let $u_{x,r}=\rho_{R_1}(3r)+g(x)+\max\{2\|g\|_{L^{\infty}(\mathbb R^n)},\frac{\|I(\cdot,-\|g\|_{L^{\infty}(\mathbb R^n)},0)\|_{L^\infty(\Omega)}}{\epsilon_7}\}\varphi_{x,r}$ where $\varphi_{x,r}$ and $\epsilon_7$ are given in Lemma $\ref{lem:bar4}$. It is obvious that $u_{x,r}(x)=\rho_{R_1}(3r)+g(x)$, $u_{x,r}\geq g$ in $\mathbb R^n$ and 
\begin{equation*}
M_{\mathcal{L}}^+u_{x,r}+C_0|\nabla u_{x,r}|\leq -\|I(\cdot,-\|g\|_{L^{\infty}(\mathbb R^n)},0)\|_{L^{\infty}(\Omega)}\quad\text{in $\Omega$.}
\end{equation*}

Now we define $\tilde u=\inf_{x\in\partial\Omega,0<r<r_\Omega}\{u_{x,r}\}$. Therefore, $\tilde u=g$ in $\partial\Omega$ and $\tilde u\geq g$ in $\mathbb R^n$. For any $x\in\partial\Omega$ and $y\in\mathbb R^n$, we have $g(y)-g(x)\leq \tilde u(y)-\tilde u(x)= \tilde u(y)-g(x)\leq \rho_{R_1}(3r)+\max\{2\|g\|_{L^{\infty}(\mathbb R^n)},\frac{\|I(\cdot,-\|g\|_{L^{\infty}(\mathbb R^n)},0)\|_{L^\infty(\Omega)}}{\epsilon_7}\}\varphi_{x,r}(y)$ for any $0<r<r_\Omega$. Therefore, $\tilde u$ is continuous on $\partial \Omega$. For any $y\in\Omega$, we define $d_y={\rm dist}(y,\partial\Omega)>0$. If $r<\frac{d_y}{2}$, then we have, for any $z\in B_{\frac{d_y}{2}}(y)$,
\begin{equation*}
u_{x,r}(z)=\rho_{R_1}(3r)+g(x)+2\max\{2\|g\|_{L^{\infty}(\mathbb R^n)},\frac{\|I(\cdot,-\|g\|_{L^{\infty}(\mathbb R^n)},0)\|_{L^\infty(\Omega)}}{\epsilon_7}\},\quad \text{for any $x\in\partial\Omega$.}
\end{equation*}
Thus, we have, for any $z\in B_{\frac{d_y}{2}}(y)$,
\begin{equation*}
\inf_{x\in\partial\Omega,\frac{d_y}{2}<r<r_\Omega}\{u_{x,r}(z)-u_{x,r}(y),0\}\leq \tilde u(z)-\tilde u(y)\leq \sup_{x\in\partial\Omega,\frac{d_y}{2}<r<r_\Omega}\{u_{x,r}(z)-u_{x,r}(y),0\}.
\end{equation*}
Since $\{u_{x,r}\}_{x\in\partial\Omega,\frac{d_y}{2}<r<r_{\Omega}}$ has a uniform modulus of continuity, $\tilde u$ is continuous in $\Omega$. Therefore, $\tilde u$ is a bounded continuous function in $\bar \Omega$. By Lemma $\ref{lem:sta}$, we have $M_{\mathcal{L}}^+\tilde u+C_0|\nabla \tilde u|\leq -\|I(\cdot,-\|g\|_{L^{\infty}(\mathbb R^n)},0)\|_{L^{\infty}(\Omega)}$ in $\Omega$. 

Now we define 
\begin{equation*}
\bar u:=\left\{\begin{array}{ll} \tilde u,
\quad \text{in $\Omega$},\\
g,\quad \text{in $\Omega^c$}.
\end{array}
  \right.
\end{equation*}
By the properties of $\tilde u$, we have $\bar u$ is a bounded continuous function in $\mathbb R^n$, $\bar u=g$ in $\Omega^c$ and $M_{\mathcal{L}}^+\bar u+C_0|\nabla \bar u|\leq -\|I(\cdot,-\|g\|_{L^{\infty}(\mathbb R^n)},0)\|_{L^{\infty}(\Omega)}$ in $\Omega$. Using {\rm (A2)} and uniform ellipticity, we have, for any $x\in\Omega$, 
\begin{eqnarray*}
I(x,-\|g\|_{L^{\infty}(\mathbb R^n)},0)-I(x,\bar u(x),\bar u(\cdot))&\leq&I(x,\bar u(x),0)-I(x,\bar u(x),\bar u(\cdot))\\
&\leq&M_{\mathcal{L}}^+\bar u(x)+C_0|\nabla \bar u(x)|\\
&\leq& -\|I(\cdot,-\|g\|_{L^{\infty}(\mathbb R^n)},0)\|_{L^\infty(\Omega)}.
\end{eqnarray*}
Thus, $I(x,\bar u(x),\bar u(\cdot))\geq 0$ in $\Omega$. 
\end{proof}
Now we have enough ingredients to conclude
\begin{theorem}\label{mainthm1}
Let $\Omega$ be a bounded domain satisfying the uniform exterior ball condition. Assume that $0<\sigma<2$, $I(x,0,0)$ is bounded in $\Omega$ and $g$ is a bounded continuous function. Assume that $I$ is uniformly elliptic and satisfies {\rm (A0)}, {\rm (A2)}. Then $(\ref{eq:gennon})$ admits a viscosity solution $u$. 
\end{theorem}
\begin{proof}
The result follows from Theorem \ref{thm:per}, Theorem $\ref{cor:hol1}$ and Theorem $\ref{thm:supsub1}$.
\end{proof}
\begin{corollary}\label{th5.4}
Let $\Omega$ be a bounded domain satisfying the uniform exterior ball condition. Assume that $0<\sigma<2$, $b_{ab}\equiv 0$ in $\Omega$ if $\sigma<1$ and $c_{ab}\geq 0$ in $\Omega$. Assume that $g$ is a bounded continuous function in $\mathbb R^n$, $\{K_{ab}(\cdot,z)\}_{a,b,z}$, $\{b_{ab}\}_{a,b}$, $\{c_{ab}\}_{a,b}$, $\{f_{ab}\}_{a,b}$ are sets of uniformly bounded and continuous functions in $\Omega$, uniformly in $a\in\mathcal{A}$, $b\in\mathcal{B}$, and $\{K_{ab}(x,\cdot):x\in\Omega,a\in\mathcal{A},b\in\mathcal{B}\}$ are kernels satisfying {\rm (H0)}-{\rm (H3)}. Then $(\ref{eq:belisa})$ admits a viscosity solution $u$.
\end{corollary}

\subsection{Degenerate case}

In the degenerate case, it is natural to construct a sub/supersolution only for $(\ref{eq:belisa})$ when $c_{ab}\geq \gamma$ for some $\gamma>0$. We remind you that $\Omega$ is a bounded domain satisfying the uniform exterior ball condition with a uniform radius $r_{\Omega}$ and, for any $x\in\partial\Omega$ and $0<r\leq r_\Omega$, $y_x^r$ is a point satisfying $\bar B_r(y_x^r)\cap\bar\Omega=\{x\}$. From now on, we will hide the dependence on $x$ for all variables and functions to make the notation simpler. For example, we will let $y^r:=y_x^r$.  For any $x\in\partial\Omega$, $y\in\Omega$ and $0<r\leq r_\Omega$, we let 
\begin{equation*}
n:=\frac{x-y^r}{|x-y^r|},\quad n_{y}^r:=\frac{y-y^r}{|y-y^r|},\quad \text{and}\quad v_{\alpha}^r(y):=\left(\left(\frac{\left(y-y^r\right)\cdot n}{r}-1\right)^+\right)^{\alpha}
\end{equation*}
(See Figure 1).
\begin{figure}[!h]
\centering
\includegraphics[width=8.5cm]{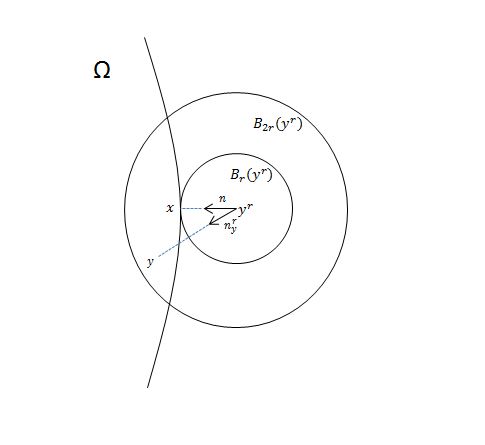}\\
\caption{\scriptsize{The exterior ball centered at $y^r.$}}
\end{figure}

Instead of letting $\{K_{ab}(x,\cdot);x\in\Omega, a\in\mathcal{A},b\in\mathcal{B}\}$ satisfy {\rm (H3)}, we let the set of kernels satisfy the following weaker assumption:\\ 
($\overline{{\rm H3}}$) There exist $C_4>0$, $0<r_1<r_{\Omega}$, $\lambda>0$ and $\mu>0$ such that, for any $x\in\partial \Omega$, $0<r<r_1$ and $y\in \Omega\cap B_{2r}(y^r)$, there is a set $A_{y}^r$ satisfying
\begin{itemize}
\item[(i)] $A_{y}^r\subset \{z;z_{n_{y}^r}<-rs_{y}^r\}\cap (B_{C_4rs_{y}^r}\setminus B_{rs_{y}^r})$ where $z_{n_{y}^r}:=z\cdot n_{y}^r$ and $s_{y}^r:=\frac{|y-y^r|}{r}-1$;
\item[(ii)] $|A_{y}^r|\geq\mu|B_{rs_{y}^r}|$;
\item[(iii)] $K(y,z)\geq (2-\sigma)\lambda(rs_{y}^r)^{-n-\sigma}$ for any $z\in A_{y}^r$.
\end{itemize}
\begin{lemma}
Suppose that $\{K_{ab}(x,\cdot);a\in\mathcal{A},b\in\mathcal{B},x\in\{y\in\Omega;{\rm dist}(y,\partial\Omega)<r_1\}\}$ satisfies $({\rm{H3}})$ for some $r_1\in(0,r_\Omega)$. Then $(\overline{{\rm H3}})$ holds for the set of kernels.
\end{lemma}
\begin{proof}
For any $x\in\partial\Omega$, $0<r<r_1$ and $y\in\Omega\cap B_{2r}(y^r)$, we define
\begin{equation}\label{eeq5.5}
\mu_{C_4}:=\frac{|(B_{C_4rs_{y}^r}\setminus B_{\frac{C_4rs_{y}^r}{2}})\cap \{z;|z_{n_{y}^r}|\leq rs_{y}^r\}|}{|B_{C_4rs_{y}^r}\setminus B_{\frac{C_4rs_{y}^r}{2}}|}.
\end{equation}
We notice that the right hand side of $(\ref{eeq5.5})$ depends only on $C_4$. It is obvious that 
\begin{equation*}
\lim_{C_4\to +\infty}\mu_{C_4}=0.
\end{equation*}
By $({\rm{H3}})$, there exists a set $A$ satisfying
\begin{equation*}
A\subset B_{C_4rs_{y}^r}\setminus B_{\frac{C_4rs_{y}^r}{2}},\quad A=-A,\quad |A|\geq \mu |B_{C_4rs_{y}^r}\setminus B_{\frac{C_4rs_{y}^r}{2}}|,
\end{equation*}
and, for any $z\in A$,
\begin{equation*}
K(y,z)\geq (2-\sigma)\lambda (\frac{C_4rs_{y}^r}{2})^{-n-\sigma}=(2-\sigma)\lambda(\frac{C_4}{2})^{-n-\sigma}(rs_{y}^r)^{-n-\sigma}:=(2-\sigma)\bar\lambda(rs_{y}^r)^{-n-\sigma}.
\end{equation*}
There exists a sufficiently large constant $C_4(\geq 2)$ such that $\mu_{C_4}<\frac{\mu}{2}$. Then
\begin{equation*}
\frac{|\{z;|z_{n_{y}^r}|> rs_{y}^r\}\cap A|}{|B_{C_4 rs_{y}^r}\setminus B_{\frac{C_4rs_{y}^r}{2}}|}\geq \frac{|A|-|(B_{C_4rs_{y}^r}\setminus B_{\frac{C_4rs_{y}^r}{2}})\cap \{z;|z_{n_{y}^r}|\leq rs_{y}^r\}|}{|B_{C_4rs_{y}^r}\setminus B_{\frac{C_4rs_{y}^r}{2}}|}\geq \frac{\mu}{2}.
\end{equation*}
Let $A_{y}^r:=A\cap\{z;z_{n_{y}^r}<-rs_{y}^r\}$. By the symmetry of $A$, we have
\begin{equation*}
|A_{y}^r|\geq \frac{\mu}{4}|B_{C_4rs_{y}^r}\setminus B_{\frac{C_4rs_{y}^r}{2}}|\geq \frac{\mu}{4} |B_{rs_{y}^r}|:=\bar\mu |B_{rs_{y}^r}|.
\end{equation*}
Therefore, $(\overline{{\rm H3}})$ holds for the set of kernels with $C_4$, $r_1$, $\bar\lambda$ and $\bar\mu$.
\end{proof}
\begin{lemma}\label{lem:bar5}
Assume that $0<\sigma<2$ and $\{K_{ab}(x,\cdot);x\in\Omega, a\in\mathcal{A},b\in\mathcal{B}\}$ are kernels satisfying {\rm(H0)}-{\rm(H2)}, ($\overline{{\rm H3}}$). Then there exists a sufficiently small $\alpha>0$ such that, for any $x\in\partial \Omega$, $0<r<r_1$ and $s\in \{l\in(0,1); y^r+(1+l)r n\in\Omega\}$, we have $I_{ab}[y^r+(1+s)r n,v_{\alpha}^r]\leq -\epsilon_8r^{-\sigma}s^{\alpha-\sigma}$ where $\epsilon_8$ is some positive constant.
\end{lemma}
\begin{proof}
We only prove the result for the case $0<\sigma<1$. For the rest of cases, the proofs are similar to those in Lemma $\ref{lem:bar1}$. For any $x\in\partial \Omega$, $0<r<r_1$ and $s\in \{l\in(0,1); y^r+(1+l)r n\in\Omega\}$, we have
\begin{eqnarray*}
&&I_{ab}[y^r+(1+s)r n,v_{\alpha}^r]\\
&=&\int_{\mathbb R^n}\left(v_{\alpha}^r(y^r+(1+s)r n+z)-v_{\alpha}^r(y^r+(1+s)r n)\right)K_{ab}(y^r+(1+s)r n,z)dz\\
&=&\int_{\mathbb R^n}\left[\left(\left(s+\frac{\tilde z_{n}}{r}\right)^+\right)^{\alpha}-s^{\alpha}\right]K_{ab}(y^r+(1+s)r n,z)dz\\
&=&r^{-\sigma}s^{\alpha-\sigma}\int_{\mathbb R^n}\left[\left(\left(1+\tilde z_{n}\right)^+\right)^{\alpha}-1\right](rs)^{n+\sigma}K_{ab}(y^r+(1+s)r n,rsz)dz\\
&=&r^{-\sigma}s^{\alpha-\sigma}\Big\{\int_{\tilde z_{n}>-1}\Big[\left(1+\tilde z_{n}\right)^{\alpha}-1\Big](rs)^{n+\sigma}K_{ab}(y^r+(1+s)r n,rsz)dz\\
&&\qquad\qquad\,\,-\int_{\tilde z_{n}\leq-1}(rs)^{n+\sigma}K_{ab}(y^r+(1+s)r n,rsz)dz \Big\},
\end{eqnarray*}
where $\tilde z_{n}:=z\cdot n$. Using ($\overline{{\rm H3}}$), we have
\begin{eqnarray*}
&&\int_{\tilde z_{n}\leq-1}(rs)^{n+\sigma}K_{ab}(y^r+(1+s)r n,rsz)dz\\
&=&(rs)^{\sigma}\int_{\tilde z_{n}\leq -rs}K_{ab}(y^r+(1+s)r n,z)dz\\
&\geq&(rs)^{\sigma}\int_{A_{y^r+(1+s)rn}^r}K_{ab}(y^r+(1+s)r n,z)dz\\
&\geq&(2-\sigma)\lambda\mu(rs)^{-n}|B_{rs}|:=2\epsilon_8.
\end{eqnarray*}
We notice that the kernel $(rs)^{n+\sigma}K_{ab}(y^r+(1+s)rn,rs\cdot)$ still satisfies {\rm (H1)} and {\rm (H2)}. By a similar calculation to $(\ref{eq::5.2})$, we have
\begin{eqnarray*}
\int_{\tilde z_{n}>-1}\Big[\left(1+\tilde z_{n}\right)^{\alpha}-1\Big](rs)^{n+\sigma}K_{ab}(y^r+(1+s)r n,rsz)dz\leq \epsilon(\alpha),
\end{eqnarray*}
where $\epsilon(\alpha)$ is a positive constant satisfying that $\epsilon(\alpha)\to 0$ as $\alpha\to 0$. Then there exists a sufficiently small $\alpha$ such that
\begin{equation*}
I_{ab}[y^r+(1+s)r n,v_{\alpha}^r]\leq -\epsilon_8r^{-\sigma}s^{\alpha-\sigma}.
\end{equation*}
\end{proof}

\begin{lemma}\label{lem:bar6}
Assume that $0<\sigma<2$, and $b_{ab}\equiv 0$ in $\Omega$ if $\sigma<1$. Assume that $\{b_{ab}\}_{a,b}$ are sets of uniformly bounded functions in $\Omega$ and $\{K_{ab}(x,\cdot);x\in\Omega, a\in\mathcal{A},b\in\mathcal{B}\}$ are kernels satisfying {\rm(H0)}-{\rm(H2)}, ($\overline{{\rm H3}}$). Then there are $\alpha>0$ and $0<s_0<1$ sufficiently small so that, for any $x\in\partial\Omega$ and $0<r<r_1$, the function 
\begin{equation*}
u_{\alpha}^r(y):=\left(\left(\frac{|y-y^r|}{r}-1\right)^+\right)^{\alpha}
\end{equation*}
satisfies, for any $a\in\mathcal{A}$ and $b\in\mathcal{B}$,
\begin{equation*}
-I_{ab}[y,u_{\alpha}^r]+b_{ab}(y)\cdot \nabla u_{\alpha}^r(y)\geq 1\quad\text{in $\Omega\cap(\bar B_{(1+s_0)r}(y^r)\setminus \bar B_r(y^r))$}.
\end{equation*}
\end{lemma}
\begin{proof}
Note that, $\forall s>0$, $u_{\alpha}^r(y^r+(1+s)rn)=v_{\alpha}^r(y^r+(1+s)rn)$, $\nabla u_{\alpha}^r(y^r+(1+s)rn)=\nabla v_{\alpha}^r(y^r+(1+s)rn)$ and that
\begin{equation*}
\left|\left(\frac{|(1+s)rn+z|}{r}-1\right)^+-\left(s+\frac{\tilde z_{n}}{r}\right)^+\right|\leq C\frac{|z-\tilde z_{n}|^2}{r^2},\quad\text{for any $z\in B_r$.}
\end{equation*}
Thus, we have
\begin{equation*}
0\leq (u_{\alpha}^r-v_{\alpha}^r)(y^r+(1+s)rn+z)\leq\left\{\begin{array}{ll} Cs^{\alpha-1}\frac{|z-\tilde z_{n}|^2}{r^2},
\quad z\in B_\frac{rs}{2},\\
C\frac{|z-\tilde z_{n}|^{2\alpha}}{r^{2\alpha}},\qquad\,\,\, z\in B_r\setminus B_{\frac{rs}{2}},\\
C\frac{|z|^{\alpha}}{r^\alpha}, \qquad\qquad\,\,\,\, z\in \mathbb R^n\setminus B_r.
\end{array}
  \right.
\end{equation*}
Using {\rm (H1)}, we have, for any $0<\sigma<2$, $a\in\mathcal{A}$, $b\in\mathcal{B}$ and $s\in \{l\in(0,1); y^r+(1+l)r n\in\Omega\}$,
\begin{eqnarray*}
0&\leq& I_{ab}[y^r+(1+s)r n,u_{\alpha}^r-v_{\alpha}^r]\\
&\leq& \int_{\mathbb R^n}(u_{\alpha}^r-v_{\alpha}^r)(y^r+(1+s)r n+z)K_{ab}(y^r+(1+s)r n,z)dz\\
&\leq&C\Big(\int_{B_{\frac{rs}{2}}}s^{\alpha-1}\frac{|z-\tilde z_{n}|^2}{r^2}K_{ab}(y^r+(1+s)r n,z)dz\\
&&\qquad+\int_{B_{\frac{r}{2}}\setminus B_{\frac{rs}{2}}}\frac{|z-\tilde z_{n}|^{2\alpha}}{r^{2\alpha}}K_{ab}(y^r+(1+s)r n,z)dz\\
&&\qquad+\int_{\mathbb R^n\setminus B_{r}}\frac{|z|^{\alpha}}{r^\alpha}K_{ab}(y^r+(1+s)r n,z)dz\Big)\\
&\leq&C\Big(\int_{B_{\frac{rs}{2}}}s^{\alpha-1}\frac{|z|^2}{r^2}K_{ab}(y^r+(1+s)r n,z)dz+\int_{\mathbb R^n\setminus B_{\frac{rs}{2}}}\frac{|z|^{2\alpha}}{r^{2\alpha}}K_{ab}(y^r+(1+s)r n,z)dz\Big)\\
&\leq&C\Lambda r^{-\sigma}(s^{\alpha-\sigma+1}+s^{2\alpha-\sigma}).
\end{eqnarray*}
By Lemma $\ref{lem:bar5}$, we have
\begin{eqnarray}
-I_{ab}[y^r+(1+s)r n,u_{\alpha}^r]
&\geq&-I_{ab}[y^r+(1+s)r n,v_{\alpha}^r]-I_{ab}[y^r+(1+s)r n,u_{\alpha}^r-v_{\alpha}^r]\nonumber\\
&\geq&r^{-\sigma}[\epsilon_8s^{\alpha-\sigma}-C\Lambda(s^{\alpha-\sigma+1}+s^{2\alpha-\sigma})].\label{eq:5.5}
\end{eqnarray}
For any $y\in \Omega\cap (B_{2r}(y^r)\setminus \bar B_r(y^r))$, we have
\begin{eqnarray*}
-I_{ab}[y,u_{\alpha}^r]&=&-\int_{\mathbb R^n}\delta_z u_{\alpha}^r(y)K_{ab}(y,z)dz\\
&=&-\int_{\mathbb R^n}\delta_z u_{\alpha}^r(y^r+(1+s_{y}^r)r n_{y}^r)K_{ab}(y,z)dz\\
&=&-\int_{\mathbb R^n}\delta_z u_{\alpha}^r(y^r+(1+s_{y}^r)r n)K_{ab}\left(y,\left(\frac{z}{|z|}+n_{y}^r-n\right)|z|\right)dz.
\end{eqnarray*}
Using ($\overline{{\rm H3}}$) and a similar estimate to $(\ref{eq:5.5})$, we have
\begin{equation*}
-I_{ab}[y,u_{\alpha}^r]\geq r^{-\sigma}[\epsilon_8(s_{y}^r)^{\alpha-\sigma}-C\Lambda((s_{y}^r)^{\alpha-\sigma+1}+(s_{y}^r)^{2\alpha-\sigma})].
\end{equation*}
By a similar estimate to $(\ref{eq:::5.3})$, there exists a sufficiently small constant $0<s_0<1$ such that we have, for any $y\in \Omega\cap (\bar B_{(1+s_0)r}(y^r)\setminus \bar B_r(y^r))$,
\begin{equation*}
-I_{ab}[y,u_{\alpha}^r]+b_{ab}(y)\cdot \nabla u_{\alpha}^r(y)\geq 1.
\end{equation*}
\end{proof}

\begin{lemma}\label{lem:bar7}
Assume that $0<\sigma<2$, $b_{ab}\equiv 0$ in $\Omega$ if $\sigma<1$ and $c_{ab}\geq \gamma$ in $\Omega$ for some $\gamma>0$. Assume that $\{K_{ab}(\cdot,z)\}_{a,b,z}$, $\{b_{ab}\}_{a,b}$, $\{c_{ab}\}_{a,b}$, $\{f_{ab}\}_{a,b}$ are sets of uniformly bounded and continuous functions in $\Omega$, uniformly in $a\in\mathcal{A}$, $b\in\mathcal{B}$, and $\{K_{ab}(x,\cdot);x\in\Omega, a\in\mathcal{A},b\in\mathcal{B}\}$ are kernels satisfying {\rm(H0)}-{\rm(H2)}, ($\overline{{\rm H3}}$). Then, for any $x\in \partial\Omega$ and $0<r<r_1$, there is a continuous  viscosity supersolution $\psi_{r}$ of $(\ref{eq3.5})$ such that $\psi_{r}\equiv 0$ in $\bar B_r(y^r)$, $\psi_{r}>0$ in $\bar B_r^c(y^r)$ and 
\begin{equation}\label{eq::5.6}
\psi_{r}\equiv\frac{\sup_{a\in\mathcal{A},b\in\mathcal{B}}\|f_{ab}\|_{L^{\infty}(\Omega)}+1}{\gamma}\quad \text{in $B_{(1+s_0)r}^c(y^r)$},
\end{equation}
where $s_0$ is given by Lemma $\ref{lem:bar6}$.
\end{lemma}
\begin{proof}
Without loss of generality, we assume that $0<\gamma<1$. We pick a sufficiently large $C_5>0$ such that
 \begin{equation}\label{eq:5.6}
 C_5>\frac{\sup_{a\in\mathcal{A},b\in\mathcal{B}}\|f_{ab}\|_{L^{\infty}(\Omega)}+1}{s_0^\alpha\gamma}.
 \end{equation}
We then define, for any $x\in\partial \Omega$ and $0<r<r_1$,
\begin{equation*}
\psi_{r}(y)=\min\left\{\frac{\sup_{a\in\mathcal{A},b\in\mathcal{B}}\|f_{ab}\|_{L^{\infty}(\Omega)}+1}{\gamma},C_5 u_{\alpha}^r(y)\right\}.
\end{equation*}
It is easy to verify that $\psi_{r}\equiv 0$ in $\bar B_r(y^r)$, $\psi_{r}>0$ in $\bar B_r^c(y^r)$ and $\psi_{r}$ is a continuous function in $\mathbb R^n$. Using $(\ref{eq:5.6})$, we know that 
\begin{equation*}
C_5u_{\alpha}^r\geq C_5s_0^\alpha\geq \frac{\sup_{a\in\mathcal{A},b\in\mathcal{B}}\|f_{ab}\|_{L^{\infty}(\Omega)}+1}{\gamma},\quad\text{in $B_{(1+s_0)r}^c(y^r)$.}
\end{equation*}
Therefore, $(\ref{eq::5.6})$ holds. Since $c_{ab}\geq\gamma>0$ in $\Omega$, $\frac{\sup_{a\in\mathcal{A},b\in\mathcal{B}}\|f_{ab}\|_{L^{\infty}(\Omega)}+1}{\gamma}$ is a viscosity supersolution of $(\ref{eq3.5})$ in $\Omega$. By Lemma $\ref{lem:bar6}$ and $(\ref{eq:5.6})$, we have, for any $y\in\Omega\cap(\bar B_{(1+s_0)r}(y^r)\setminus \bar B_r(y^r))$,
\begin{eqnarray}
&&\sup_{a\in\mathcal{A}}\inf_{b\in\mathcal{B}}\{-I_{ab}[y,C_5u_{\alpha}^r]+C_5b_{ab}(x)\cdot \nabla u_{\alpha}^r(y)+C_5c_{ab}(x)u_{\alpha}^r(y)+f_{ab}(y)\}\nonumber\\
&\geq&\sup_{a\in\mathcal{A},b\in\mathcal{B}}\|f_{ab}\|_{L^{\infty}(\Omega)}+1+f_{ab}(y)\geq 0.\label{eq:5.8}
\end{eqnarray}
Therefore, $\psi_{r}$ is a continuous viscosity supersolution of $(\ref{eq3.5})$ in $\Omega$. 
\end{proof}

\begin{theorem}\label{thm:supsub2}
Assume that $0<\sigma<2$, $b_{ab}\equiv 0$ in $\Omega$ if $\sigma<1$ and $c_{ab}\geq \gamma$ in $\Omega$ for some $\gamma>0$. Assume that $g$ is a bounded continuous function in $\mathbb R^n$, $\{K_{ab}(\cdot,z)\}_{a,b,z}$, $\{b_{ab}\}_{a,b}$, $\{c_{ab}\}_{a,b}$, $\{f_{ab}\}_{a,b}$ are sets of uniformly bounded and continuous functions in $\Omega$, uniformly in $a\in\mathcal{A}$, $b\in\mathcal{B}$, and $\{K_{ab}(x,\cdot);x\in\Omega, a\in\mathcal{A},b\in\mathcal{B}\}$ are kernels satisfying {\rm(H0)}-{\rm(H2)}, ($\overline{{\rm H3}}$). Then $(\ref{eq:belisa})$ admits a continuous viscosity supersolution $\bar u$ and a continuous viscosity subsolution $\underbar u$ and $\bar u=\underbar u=g$ in $\Omega^c$. 
\end{theorem}
\begin{proof}
We only prove $(\ref{eq:belisa})$ admits a viscosity supersolution $\bar u$ such that $\bar u=g$ in $\Omega^c$. Since $g$ is a continuous function, let $\rho_R$ be a modulus of continuity of $g$ in $B_R$. Let $R_1$ be a sufficiently large constant such that $\Omega\subset B_{R_1-1}$. For any $x\in\partial\Omega$, we let 
\begin{equation*}
u_{r}=\rho_{R_1}(3r)+g(x)+\left(2\|g\|_{L^{\infty}(\mathbb R^n)}\frac{\gamma}{\sup_{a\in\mathcal{A},b\in\mathcal{B}}\|f_{ab}\|_{L^{\infty}(\Omega)}+1}+1\right)\psi_{r},
\end{equation*}
where $\psi_{r}$ is given in Lemma $\ref{lem:bar7}$. Using Lemma $\ref{lem:bar7}$, $u_{r}(x)=\rho_{R_1}(3r)+g(x)$, $u_{r}\geq g$ in $\mathbb R^n$ and $u_{r}$ is a continuous viscosity supersolution of $(\ref{eq3.5})$ in $\Omega$. 
Then the rest of the proof is similar to Theorem $\ref{thm:supsub1}$.
\end{proof}

\begin{theorem}\label{thm:disvis}
Let $\Omega$ be a bounded domain satisfying the uniform exterior ball condition. Assume that $0<\sigma<2$,  $b_{ab}\equiv 0$ in $\Omega$ if $\sigma<1$ and $c_{ab}\geq \gamma$ in $\Omega$ for some $\gamma>0$. Assume that $g$ is a bounded continuous function in $\mathbb R^n$, $\{K_{ab}(\cdot,z)\}_{a,b,z}$, $\{b_{ab}\}_{a,b}$, $\{c_{ab}\}_{a,b}$, $\{f_{ab}\}_{a,b}$ are sets of uniformly bounded and continuous functions in $\Omega$, uniformly in $a\in\mathcal{A}$, $b\in\mathcal{B}$, and $\{K_{ab}(x,\cdot);x\in\Omega, a\in\mathcal{A},b\in\mathcal{B}\}$ are kernels satisfying {\rm(H0)}-{\rm(H2)}, ($\overline{{\rm H3}}$). Then $(\ref{eq:belisa})$ admits a discontinuous viscosity solution $u$.
\end{theorem}
\begin{proof}
The result follows from Corollary \ref{col:biper} and Theorem $\ref{thm:supsub2}$.
\end{proof}

\textbf{Acknowledgement.}  We would like to thank the referee for valuable comments which improved the paper.

\end{document}